\documentclass[11pt]{article}

\usepackage{amsthm}
\usepackage{amsmath}
\usepackage{amsfonts, epsfig, amsmath, amssymb, color, amscd}
\usepackage[cmtip,arrow]{xy}
\usepackage{pb-diagram, pb-xy}
\usepackage{amssymb,epsfig}
\usepackage{color}
\usepackage[cmtip,arrow]{xy}
\usepackage{pb-diagram, pb-xy}
\usepackage{amssymb,epsfig,amsfonts}

\usepackage[linktocpage=true,plainpages=false,pdfpagelabels=false]{hyperref}

\usepackage[mathscr]{eucal}
\usepackage{floatflt}
\usepackage[T1]{fontenc}
\usepackage{graphicx}
\usepackage{indentfirst}

\usepackage{url}

\usepackage{mathtools}

\usepackage{manfnt}
\usepackage{color}

\usepackage{tikz}
\usepackage{eufrak}

\usetikzlibrary{calc, positioning, shapes, fit, matrix, decorations}
\usetikzlibrary{decorations.shapes}

\usepackage[cmtip]{xypic}
\xyoption{arc}

\newfont{\sdbl}{msbm9}
\newfont{\dbl}{msbm10 at 12pt}
\theoremstyle{definition}

\newcommand{\cg}{{{\cal G}}}

\newcommand{\dz}{{\mbox{\dbl Z}}}

\newcommand{\dn}{{\mbox{\dbl N}}}

\newcommand{\ord}{\mathop{\rm ord}\nolimits}

\newcommand{\Pic}{\mathop {\rm Pic}}

\newcommand{\Sup}{\mathop {\rm Supp}}

\newcommand{\rk}{\mathop {\rm rk}}

\newcommand{\Projj}{\mathop {\rm Proj}}

\newcommand\limind{\mathop{\underrightarrow{\lim}}}

\newcommand{\idm}{\mathfrak{m}}

\makeatother
\newtheorem{Def}{Definition}[section]
\newtheorem{rem}{Remark}[section]

\theoremstyle{plain}
\newtheorem{Prop}{Proposition}[section]
\newtheorem{theorem}{Theorem}[section]
\newtheorem{lemma}{Lemma}[section]

\numberwithin{equation}{section}

\newcommand{\co}{{{\cal O}}}
\newcommand{\cf}{{{\cal F}}}

\newcommand{\cs}{{{\cal S}}}
\newcommand{\cm}{{{\cal M}}}

\newcommand{\Ord}{\mathop {\rm \bf ord}}
\newcommand{\crr}{{{\cal R}}}

\begin{document}
\title{Normal forms for ordinary differential operators, III}
\author{Junhu Guo \and A.B. Zheglov}
\date{}
\maketitle

\begin{abstract}
In this paper, which is  a follow-up of our first paper "Normal forms for ordinary differential operators, I", we extend the  explicit parametrisation of torsion free rank one sheaves on projective irreducible curves with vanishing cohomology groups obtained earlier to analogous parametrisation of torsion free sheaves of arbitrary rank with vanishing cohomology groups  on projective irreducible curves.

As an illustration of our theorem we calculate one explicit example of such parametrisation, namely for rank two sheaves on the Weierstrass cubic curve.   

\end{abstract}

\tableofcontents

\section{Introduction}

This paper is a follow-up of \cite{GZ24} and we use its notation. We collect all necessary notation in the list \ref{S:list} below, for details we refer to  \cite{GZ24} or to \cite[Sec. 2]{Zheglov25} for a short exposition of results from \cite{GZ24}.

In \cite{GZ24} we developed the generic theory of normal forms for ordinary differential operators, which was conceived and developed as a part of the generalised Schur theory offered in \cite{A.Z}, and applied it to obtain two applications in different directions of algebra/algebraic geometry. The first application was a new explicit parametrisation of torsion free rank one sheaves on projective irreducible curves with vanishing cohomology groups. 
The second application was a commutativity criterion for operators in the Weyl algebra or, more generally, in the ring of ordinary differential operators.

In this paper we extend the explicit description of torsion free rank one sheaves on projective irreducible curves to the torsion free sheaves with vanishing cohomology groups of arbitrary rank. As in the work \cite{GZ24}, we concentrate only on explicit parametrization and do not touch upon the question of representability of the module functor, leaving it for future work. In the second part of the paper we give in details one explicit example of such parametrisation, namely for rank two sheaves on the Weierstrass cubic curve.  
The description of such sheaves is known via the Fourier-Mukai transform, and can be found e.g. in \cite{BZ}, where this description was used to give a solution to the problem of Previato and Wilson  \cite{PW}. Recall that this is a problem of explicit description of the spectral sheaf via coefficients of commuting differential operators of orders 4 and 6 with a given spectral curve.  

To show how our parametrisation fit with calculations from \cite{BZ}, we calculate it via corresponding commuting operators of order 4 and 6 -- up to now the most well-studied explicit examples of operators of rank >1 from the work \cite{KN}, cf. also \cite{Grun}, \cite{G}, \cite{Mokh}. As a byproduct we illustrate many statements of the theory of normal forms from \cite[\S 2]{GZ24}. At the end we get a dictionary between two descriptions of spectral sheaves - one given in \cite{BZ}, and another given in terms of coefficients of normal forms from our main theorem \ref{T:parametrisation}.

To formulate our main theorem we need to introduce some notation first. Suppose $K$ is a field of characteristic zero, and $\tilde{K}$ is an algebraic closure of $K$.

\begin{Def}
\label{D:affine_spec_curve}
We'll call an affine curve $C_0$ over $K$ as {\it affine spectral curve} if it can be compactified with the help of one regular $K$-point, i.e. if there exists a projective curve $C$ over $K$ and a regular $K$-point $p$ such that $C_0\simeq C\backslash p$.
\end{Def}

Given an affine spectral curve, we can always choose a system of generators $w_1, \ldots , w_m$ of the ring $\co_{C}(C_0)$ such that the pole order (at the point $p$) of $w_1$ is coprime with the pole orders of other generators, and the images of all $w_i$ under the Krichever map after some choice of $\pi$ (see details in section \ref{S:classif_thms}),  are monic (cf. \cite[Rem 3.7]{GZ24}). 

Assume now $P,Q\in D_1=K[[x]][\partial ]$ are differential operators of positive order, where $Q$ is normalized and $P$ is monic and $[P,Q]=0$. Considering the ring of ordinary differential operators $D_1$ as a subring of a certain complete non-commutative ring $\hat{D}_1^{sym}$ (appeared first in \cite{BurbanZheglov2017} and then extensively studied in \cite{A.Z} and \cite{GZ24}), a notion of {\it a normal form} of $P$ with respect to $Q$ was offered in \cite{GZ24}. Recall that for given regular operators $P,Q\in \hat{D}_1^{sym}$ a normal form of of $P$ with respect to $Q$ is an operator  $P'=SPS^{-1}$, where $S$ is a  Schur operator $S\in \hat{D}_1^{sym}$, corresponding to the operator $Q$, i.e. $S$ is invertible, with $\Ord (S)=0$, and $SQS^{-1}=\partial^q$, where $q=\Ord (Q)$. If $P,Q$ are differential operators (i.e. $P,Q\in D_1$), then they are regular iff they have invertible highest coefficients (cf. \cite[Rem. 2.2]{Zheglov25}); in this case the Schur operator satisfies some additional properties (see \cite[\S 2.4]{GZ24}). If $[P,Q]=0$, then $P'\in C(\partial^q)\subset \hat{D}_1^{sym}$, and therefore can be presented by a matrix from the ring $M_q(K[D^q])$, see \cite[\S 3]{GZ24}. The normal form $P'$ (and the Schur operator $S$) is not uniquely defined, but up to conjugation (or, correspondingly, multiplication) by invertible zeroth order operators from the centralizer $C(\partial^q)$. If the orders of operators $P,Q$ are coprime, then this non-uniqueness can be fixed by choosing a normalised normal form, which is uniquely defined not only in its conjugacy class by invertible zeroth order operators from the centralizer, but also in its conjugacy class by arbitrary invertible operators from the centralizer (see \cite[Lem. 3.4]{GZ24}). 

If the orders of operators $P,Q$ are not coprime, the situation becomes more complicated.  We'll call a normal form $P'$ of the operator $P$ with respect to $Q$ as {\it partially normalized}, if it satisfies conditions of lemma \ref{L:normalisation} below (of course, if the orders of operators $P,Q$ are coprime, then a partially normalized normal form is the normalized form from \cite[Lem. 3.4]{GZ24}). 

Extending the notion of the  normal form for a pair of operators to the case of a finitely generated commutative ring, and following the exposition of \cite[\S 3]{GZ24}, let's introduce a notion of a partially normalised normal form of rank $r$. Spectral sheaves of rank $r$ will be parametrised in terms of its coefficients.  

\begin{Def}
\label{D:normalised_normal_form}
We'll call a commutative ring $B'\in C(\partial^q)\subset \hat{D}_1^{sym}$ generated over $K$ by monic operators $P_1'=\partial^{rq}, P_2', \ldots, P_m'$, where $r=GCD(rq, \Ord (P_2'))=\rk B':=GCD(\Ord (P')|\quad P'\in B')$ and $P_2'$ is partially normalised in the sense of lemma \ref{L:normalisation}, as a {\it partially normalised normal form} of rank $r$ with respect to a  set of generators $P_1',\ldots , P_m'$, and we'll call the coefficients of the operators $P_i'$ as coordinates of $B'$. 
\end{Def}

Clearly, 
$$
B'\simeq K[P_1',\ldots ,P_m']\simeq K[T_1,\ldots ,T_m]/I,
$$
where $I=(f_1,\ldots ,f_k)$ is a prime ideal, $f_i\in K[T_1,\ldots ,T_m]$. 

\begin{Def}
\label{D:X_B}
We define the affine algebraic set $X_{[B']}$ as a set determined by equations on coordinates of $B'$ coming from equations $f_i(P_1', \ldots , P_m')=0$ for all $i$ and from commutation relations between $P_2', \ldots , P_m'$. 

We define the equivalence relation on the closed points of this set over $K$: two points are equivalent if the corresponding operators in matrix form (i.e. $\psi\circ \hat{\Phi} (P_i')$)  
are conjugated with the help of an invertible block-diagonal matrix (one for all)
$$
\left (
\begin{array}{cccc}
T&0&\ldots &0\\
0&T&\ldots &0\\
\vdots&&\ddots&0\\
0&\ldots &&T\\
\end{array}
\right )
$$
 with  $T\in GL_r(K)$ (this is an  
invertible operator from $C(\partial^r)\subset C(\partial^{rq})\subset \hat{D}_1^{sym}$ of order $<r$) and partial normalisation (conjugation with the help of invertible zeroth order monic operators from $C(\partial^q)$ from lemma \ref{L:normalisation} -- in matrix form such operators are represented as low-triangular matrices with constant coefficients). 

We'll denote by $[x]$, $x\in X_{[B']}$ the equivalence classes of closed points in $X_{[B']}$. 
\end{Def}

\begin{theorem}
\label{T:parametrisation}
Let $C_0$ be an affine spectral curve over $K$ and let $C$ be its one-point compactification, $C_0=C\backslash p$. Assume 
$$
\co_C(C_0)\simeq K[w_1,\ldots ,w_m]\simeq K[T_1,\ldots ,T_m]/I,
$$
where $I=(f_1,\ldots ,f_k)$ is a prime ideal and the order of $w_1$ is coprime with the orders of $w_i$, $i>1$,  and the images of $w_i$ under the Krichever map (after some choice of $\pi$) are monic.

Then there exist partially normalised normal forms of rank $r$ $B'\simeq \co_C(C_0)$ with respect to the ordered set of generators $P_1'=\partial^{rq}, \ldots , P_m'$, corresponding to generators $w_i$ via this isomorphism, where $\Ord (P_i')=r\ord (w_i)$ for all $i\ge 1$, and there is a one to one correspondence between equivalence classes of closed $K$-points $[x]$, $x\in X_{[\co_C(C_0)]}$, and isomorphism classes of torsion free rank $r$ sheaves $\cf$ on $C$ with vanishing cohomology groups $H^0(C,\cf )= H^1(C,\cf )=0$.
\end{theorem}

It would be interesting to adopt the theory of normal forms to get a similar description of data corresponding to commuting fractional differential operators from recent paper \cite{Casper} or to get  applications concerning the bispectrality phenomena, cf. \cite{BHY},  \cite{Casper2}.

The structure of this article is the following. In section \ref{S:prelim} we recall all necessary definitions and constructions from the classification theory of rings of commuting ordinary differential operators of arbitrary rank. 
In section \ref{S:moduli} we prove our main theorem  \ref{T:parametrisation}. In section \ref{S:example} we give in details the explicit example of parametrisation of torsion free sheaves of rank two with vanishing cohomology groups over a Weierstrass cubic curve. This section is divided into subsections according to various steps of calculations and different cases from \cite{BZ}. First we recall a description of spectral sheaves on a Weierstrass cubic curve and their description in terms of coefficients of commuting operators of orders 4 and 6 from \cite{BZ}. Then, in subsection \ref{S:normal_form_ofL_6} we calculate a generic normal form of the operator $L_6$ (of order 6) with respect to the operator $L_4$. In subsection \ref{S:self-adjoint} we calculate the partially normalised normal forms in the sense of lemma \ref{L:normalisation} for this operator in the case of a self-adjoint operator $L_4$, according to the list from \cite{BZ}, and their matrix presentations.  In subsection \ref{S:non-self_adj} we do the same in the case of a non-self-adjoint operator $L_4$. At last, in subsection \ref{S:last_subsection} we compare these matrix presentation for isomorphic sheaves and find the conjugating matrices for them, confirming the statement of theorem \ref{T:parametrisation}. 

{\bf Acknowledgements.}  We are grateful to the Sino-Russian Mathematics Center at Peking University for hospitality and excellent working conditions while preparing this paper.

The work of A. Zheglov is an output of a research project implemented as part of the Basic Research Program at the National Research University Higher School of Economics (HSE University).

The work of J. Guo was partially supported by the National Key R and D Program of China (Grant No. 2020YFE0204200).

\subsection{List of notations}
\label{S:list}

 Here we recall  the most important notations used in this paper from \cite{GZ24}.
 
1. $K$ is a field of characteristic zero. Recall some notation from \cite{A.Z}:
	$\hat{R}:=K [[x_1,\ldots ,x_n]]$, the $K$-vector space 
$$
\cm_n := \hat{R} [[\partial_1, \dots, \partial_n]] = \left\{
\sum\limits_{\underline{k} \ge \underline{0}} a_{\underline{k}} \underline{\partial}^{\underline{k}} \; \left|\;  a_{\underline{k}} \in \hat{R} \right. \;\mbox{for all}\;  \underline{k} \in \dn_0^n
\right\},
$$
$\upsilon:\hat{R}\rightarrow \dn_0\cup \infty$ --  the discrete valuation defined by the unique maximal ideal $\idm = (x_1, \dots, x_n)$ of $\hat{R}$,  \\
for any element
$
0\neq P := \sum\limits_{\underline{k} \ge \underline{0}} a_{\underline{k}} \underline{\partial}^{\underline{k}} \in \cm_n
$
$$
\Ord (P) := \sup\bigl\{|\underline{k}| - \upsilon(a_{\underline{k}}) \; \big|\; \underline{k} \in \dn_0^n \bigr\} \in \dz \cup \{\infty \},
$$
$$
\hat{D}_n^{sym}:=\bigl\{Q \in \cm_n \,\big|\, \Ord (Q) < \infty \bigr\};
$$
$
P_m:= \sum\limits_{ |\underline{i}| - |\underline{k}| = m} \alpha_{\underline{k}, \underline{i}} \,  \underline{x}^{\underline{i}} \underline{\partial}^{\underline{k}}
$ -- the $m$-th \emph{homogeneous component} of $P$,\\
$\sigma (P):=P_{\Ord (P)}=P_{-d}$ -- the highest symbol.
	
2. In this paper we use: 	
	$\hat{R}:=K[[x]]$, $D_1:=\hat{R}[\partial]$, 
$$\hat{D}_1^{sym}:=\{Q=\sum_{k\ge 0}a_k\partial^k|\Ord(Q)<\infty\}.$$ 

Operators:  $\delta:=\exp((-x)\ast \partial)$,   $\int:=(1-exp((-x)\ast \partial))\cdot \partial^{-1} $,  	$A_{k;i}:=\exp((\xi^{i}-1)x\ast \partial)\in \hat{D}^{sym}_{1}\hat{\otimes}_{K}\tilde{K}$ (in the case when $k$ is fixed, simply written as $A_i$), where $\tilde{K}=K[\xi]$, $\xi$ is a primitive $k$th root of unity, 
$\Gamma_i=(x\partial)^i$. $B_n=\frac{1}{(n-1)!}x^{n-1}\delta\partial^{n-1}$. 

$\hat{D}^{sym}_{1}\hat{\otimes}_{K}\tilde{K}$ means the same ring $\hat{D}^{sym}_{1}$, but defined over the base field $\tilde{K}$.

Operators from $\hat{D}_1^{sym}$ written in the (Standard) form as 
	$$
	H=[\sum_{0\leq i<k}f_{i;r}(x, A_{k;i}, \partial )+\sum_{0<j\leq N}g_{j;r}B_{j}]D^{r},
	$$
	where $D^r= \partial^r$ if $r\ge 0$ and $D^r= \int^{-r}$ otherwise, 
	are called HCP and  form a sub-ring $Hcpc(k)$. Here $f_{i;r}(x, A_{k;i}, \partial)$ is a polynomial of $x, A_{k;i},\partial$,  $\Ord(f_{i;r})=0$,  of the form
		$$
		f_{i;r}(x, A_{k;i}, \partial )=\sum_{0\leq l\leq d_{i}}f_{l,i;r}x^l A_{k;i}\partial^l
		$$
		for some $d_i\in \dz_+$, where $f_{l,i;r}\in \tilde{K}$. The number  $d_i$ is called the {\it $x$-degree of $f_{i;r}$}:  $deg_{x}(f_{i;r}):= d_i$; $g_{j;r}\in \tilde{K}$, $g_{j;r}=0$ for $j\le -r$ if $r<0$.
	
	They can be written also in G-form: 
	$$
	H=(\sum_{0\leq i<k}\sum_{0\leq l\leq d_i} f'_{l,i;r}\Gamma_lA_i+\sum_{0<j\leq N}g_{j;r}B_{j})D^{r}
	$$
	The $A$ and $B$ Stable degrees of HCP are defined as 
	$$
	Sdeg_A(H)=\max \{d_i|\quad 0\leq i<k \} \quad \mbox{or $-\infty$, if all $f_{l,i;r}=0$ } 
	$$
	and 
	$$Sdeg_B(H)=\max\{j|\quad g_{j;r}\neq0\} \quad \mbox{or $-\infty$, if all $g_{j;r}=0$}.
	$$

	In the case when $Sdeg_B(HD^p)=-\infty,\forall p\in \mathbf{Z}$ $H$ is called {\it totally free of} $B_j$.
	
	An operator $P\in  \hat{D}_1^{sym}$ satisfies {\it condition $A_q(k)$}, $q,k\in \dz_+$, $q>1$ if 
	\begin{enumerate}
		\item
		$P_{t}$ is a HCP  from $Hcpc (q)$  for all $t$;
		\item
		$P_{t}$ is totally free of $B_j$ for all $t$;
		\item
		$Sdeg_A(P_{\Ord (P)-i})< i+k$ for all $i>0$;
		\item
		$\sigma (P)$ does not contain $A_{q;i}$, $Sdeg_A(\sigma (P))=k$.
	\end{enumerate}

3. The operator $P\in D_1$ is called {\it normalized} if $P=\partial^p+a_{p-2}\partial^{p-2}+\ldots $. The operator $P\in D_1$ is {\it monic} if its highest coefficient is 1. Analogously, $P\in \hat{D}_1^{sym}$ is monic if $\sigma (P)=\partial^p$.  

 $\mathfrak{B}=\crr_S$ is the right quotient ring of $\crr = \tilde{K}^{\oplus k} [D,\sigma ]$ by $S=\{D^k|k\ge 0\}$, $\mathfrak{B}\simeq  \tilde{K}^{\oplus k}[\tilde{D}, \tilde{D}^{-1}]$, with commutation relations $\tilde{D}^{-1}a=\sigma (a)\tilde{D}^{-1}$, $a\in \tilde{K}^{\oplus k}$, $\sigma (a_0, \ldots ,a_{k-1})=(a_{k-1}, a_0, \ldots , a_{k-2})$. And the ring of skew pseudo-differential operators 
	$$
	E_k:=\tilde{K}[\Gamma_1, A_1]((\tilde{D}^{-1}))=\{\sum_{l=M}^{\infty}P_l\tilde{D}^{-l} | \quad P_l\in \tilde{K}[\Gamma_1, A_1]\} \simeq \tilde{K}^{\oplus k}[\Gamma_1]((\tilde{D}^{-1}))
	$$
with the commutation relations 
$$
\tilde{D}^{-1}a=\sigma (a)\tilde{D}^{-1}, \quad a\in \tilde{K}[\Gamma_1, A_1] \quad \mbox{where \quad }
\sigma (A_1)= \xi^{-1} A_1, \quad \sigma (\Gamma_1)=\Gamma_1+1.
$$	
	
	$\widehat{Hcpc}_B(k)$ is the $\tilde{K}$-subalgebra in $\hat{D}_1^{sym}\hat{\otimes}\tilde{K}$ consisting of operators whose homogeneous components are HCPs totally free of $B_j$. 
	
	$$\Phi : \tilde{K}[A_1,\ldots ,A_{k-1}]\rightarrow \tilde{K}^{\oplus k} , \quad P\mapsto (\sum_i p_i\xi^{i}, \ldots , \sum_i p_i\xi^{i(k-1)})
$$
is an isomorphism of $\tilde{K}$-algebras. 	

The map 
$$
\hat{\Phi}: \widehat{Hcpc}_B(k) \hookrightarrow E_k
$$
defined on monomial HCPs from $\widehat{Hcpc}_B(k)$ as $\hat{\Phi} (a A_j\Gamma_iD^l):=a \Phi (A_j)\Gamma_i\tilde{D}^l$ and extended by linearity on the whole $\tilde{K}$-algebra $\widehat{Hcpc}_B(k)$, is an embedding of $\tilde{K}$-algebras.
	
	Suppose $B$ is a commutative sub-algebra of $D_1$, then $(C,p,\cf)$ stands for a part of the spectral data of $B$ (the spectral curve, point at infinity and the spectral sheaf with vanishing cohomologies).
	
	The classical ring of pseudo-differential operators is defined as 
	$$
E=K[[x]]((\partial^{-1})).
	$$
	
There is an isomorphism of $\tilde{K}$-algebras $\psi:\mathfrak{B}\rightarrow M_k(C(\mathfrak{B}))$, where $C(\mathfrak{B})\simeq \tilde{K}[\tilde{D}^k, \tilde{D}^{-k}]$, ($\tilde{K}$ is diagonally embedded into $\tilde{K}^{\oplus k}$):
	$$
	\psi\begin{pmatrix}
		h_0 \\
		h_1 \\
		\cdots \\
		h_{k-1}
	\end{pmatrix}=\begin{pmatrix}
		h_0 &  &  &  \\
		& h_1 &  &  \\
		&  & \cdots &  \\
		&  &  & h_{k-1}
	\end{pmatrix}\quad \psi(D)=T:=\begin{pmatrix}
		& 1 &  & \cdots &  \\
		&  & 1 & \cdots &  \\
		\cdots & \cdots & \cdots & \cdots & \cdots \\
		&  &  & \cdots & 1 \\
		D^k &  &  & \cdots & 
	\end{pmatrix}
	$$
	with $\psi(D^l)=T^l$, and extended by linearity. The map $\psi$ can be obviously extended to 
	$$
\psi : \tilde{K}^{\oplus k}((\tilde{D}^{-1})) \hookrightarrow M_k(\tilde{K}((\tilde{D}^{-k}))).
	$$

Elements of the centralizer $C(\partial^k)$ embedded to $\mathfrak{B}\subset E_k$ via $\hat{\Phi}$ we'll call as a {\it vector form} presentation, and the same elements embedded to $M_k(\tilde{K}[D^k])$ via $\psi\circ \hat{\Phi}$ -- as a {\it matrix form} presentation. 

Translating the description of the centralizer $C(\partial^k)$  into the vector form, we get that $\hat{\Phi}(C(\partial^k))$ consists of Laurent polynomials in 
$\tilde{D}$ with coefficients from $K^{\oplus k}$ and with additional conditions: the coefficient $s_i$ at $\tilde{D}^{-i}$, $i>0$, has a shape $s_{i,j}=0$ for $j=0, \ldots , i-1$.

\section{Preliminaries}
\label{S:prelim}

Suppose $K$ is a field of characteristic zero. 

We'll recall in this section necessary facts from the classification theory of rings of commuting ordinary differential operators of arbitrary rank (\cite{Kr1}, \cite{Kr2}, \cite{Dr}, \cite{Mumford_article}, \cite{SW}, \cite{Verdier}), following in part the exposition from \cite{Zheglov_book}, \cite{Mu}, \cite{Q}.  

Let $B\subset D_1=K[[x]][\partial ]$ be a commutative subring. It is called {\it elliptic} if it contains a monic differential operator of positive order. It is well known that all operators from an elliptic ring are {\it formally elliptic}, i.e. have constant highest coefficients. Two commutative elliptic subrings $B_1, B_2$  are called {\it equavalent} if there exists an invertible function $f\in K[[x]]^*$ such that 
$B_1=f^{-1}B_2f$. 

Each equivalence class $[B]$ contains a {\it normalised representative}, which determines the whole class uniquely. For example, we can choose a monic differential operator of minimal positive order in $B$, and normalise it (conjugating by an appropriate function $f\in \hat{R}$). The {\it rank} of a commutitive subring $B\subset D_1$ is the number
$$
r=\rk (B)=GCD\{\ord (P)|P\in B\}.
$$ 

Equivalence classes of commutative subrings $[B]$ of rank $r$ are classified in terms of isomorphism classes of algebraic-geometric spectral data and also in terms of equivalence classes of Schur pairs as follows. 

There are several versions of algebraic-geometric spectral data and Schur pairs, see e.g. \cite{Zheglov_book} for details. We'll need two purely algebraic version of them -- a modifications of data from \cite{Mu} and \cite{Q}. 

\subsection{Geometric data}
\label{S:geom_data}
 
\begin{Def}{(\cite[Def.9.25]{Zheglov_book})}
\label{D:proj_spec_data}
The projective spectral data of rank $r$ consists of
\begin{enumerate}
\item
a projective irreducible curve $C$ over $K$;
\item
a regular $K$-point $p\in C$;
\item
a torsion free coherent sheaf $\cf$ of rank $r$ such that $H^0(C, \cf )= H^1(C, \cf) =0$;
\item
an embedding of local rings $\pi :\hat{\co}_{C,p} \hookrightarrow K[[z]]$ such that $\pi (f)\cdot K[[z]]= z^r\cdot K[[z]]$;
\item
an isomorphism of $\hat{\co}_{C,p}$-modules $\hat{\phi}:\hat{\cf}_p \simeq z\cdot K[[z]]$
\end{enumerate}
\end{Def}

\begin{Def}
\label{D:proj_spec_data_isom}
Two projective spectral data $(C_1, p_1, \cf_1, \pi_1, \hat{\phi}_1)$, $(C_2, p_2, \cf_2, \pi_2, \hat{\phi}_2)$ are isomorphic if there exists an isomorphism of curves $\beta :C_1\rightarrow C_2$ and an isomorphism of sheaves $\psi :\cf_2 \rightarrow \beta_*\cf_1$   such that $\beta (p_1)=p_2$ and 
\begin{itemize}
\item
there is an automorphism $\bar{h}: K[[z]]\rightarrow K[[z]]$ of rings such that 
$$
\bar{h}=z+a_2z^2+\ldots 
$$
and the following diagram of ring homomorphisms is commutative:
\begin{equation}\label{E:comm_diagram1}
\begin{array}{c}
\xymatrix{
\hat{\co}_{C_1,p_1}   \ar[d]^-{\pi_1} & & \hat{\co}_{C_2,p_2}  \ar[ll]_-{\hat{\beta}_{p_2}^{\sharp}} \ar[d]^-{\pi_2}\\
K[[z]]  & & K[[z]] \ar[ll]_-{\bar{h}}\\
}
\end{array}
\end{equation}
\item
there is a $K[[z]]$-module isomorphism $\xi :z\cdot K[[z]] \rightarrow z\cdot K[[z]]$, where $z\cdot K[[z]]$ on the right hand side is a $h_*K[[z]]$-module, i.e. $a\cdot v=\bar{h}(a) \cdot v$ (and therefore, $\xi$ is just given by the rule $a\mapsto \bar{h}(a)\xi (1)$, $\xi (1)\in K[[z]]^*$), such that the following diagram of $\hat{\co}_{C_2,p_2}$-modules isomorphisms is commutative:
\begin{equation}\label{E:comm_diagram2}
\begin{array}{c}
\xymatrix{
\hat{\cf}_{2,p_2}  \ar[rr]^-{\hat{\psi}} \ar[d]^-{\hat{\phi}_2} & & \widehat{\beta_*\cf_1}_{p_2}\simeq \hat{\cf}_{1,p_1}  \ar[d]^-{\hat{\beta}_*(\hat{\phi}_1)}\\
z\cdot K[[z]] \ar[rr]^-{(\pi_2)_*\xi} & & z\cdot K[[z]] \\
}
\end{array}
\end{equation}
(more precisely, the isomorphisms in this diagram look as follows: for any $a\in \hat{\co}_{C_2,p_2}$ and $f\in \hat{\cf}_{2,p_2}$ we have 
\begin{multline*}
\hat{\psi}(a\cdot f)= \hat{\beta}_{p_2}^{\sharp}(a) \hat{\psi}(f) (= a\cdot \hat{\psi}(f)), \quad \\ \hat{\beta}_*(\hat{\phi}_1)(a\cdot \hat{\psi}(f))=\pi_1(\hat{\beta}_{p_2}^{\sharp}(a)) \hat{\phi}_1(\hat{\psi}(f)),
\end{multline*} 
\begin{multline*}
\hat{\phi}_2(a\cdot f)=\pi_2(a)\hat{\phi}_2(f), \quad \\
((\pi_2)_*\xi )(\pi_2(a)\hat{\phi}_2(f))=\bar{h}(\pi_2(a))(\xi (\hat{\phi}_2(f)))=\bar{h}(\pi_2(a)\hat{\phi}_2(f))\xi (1);
\end{multline*} 
the isomorphism $\widehat{\beta_*\cf_1}_{p_2}\simeq \hat{\cf}_{1,p_1}$ is also an isomorphism of $\hat{\co}_{C_2,p_2}$-modules, where $\hat{\cf}_{1,p_1}$ has a $\hat{\co}_{C_2,p_2}$-module structure via the homomorphism of local rings $\hat{\beta}_{p_2}^{\sharp}$.) 
\end{itemize}
\end{Def}

\begin{rem}
It is possible to extend the notion of a projective spectral datum  and define a {\it category} $\cg_r$ of data, whose objects are projective spectral data of rank $r$, see \cite{Mu}.
\end{rem} 

A slightly different version of data better suited for study of the relative version of the classification theory was offered by I. Quandt in \cite{Q}. We'll use a simplified version of her definition in this paper:

\begin{Def}{(\cite[Def.2.9]{Q})}
\label{Q2.9}
By a geometric data of rank $r$  we mean a tupel 
$$
(C, P,\rho ,\cf ,\hat{\varphi} )
$$
such that 
\begin{enumerate}
\item
$C$ is a projective irreducible curve $C$ over $K$;
\item
$p\subset C$ is a regular $K$-point; 
\item
$\rho :\hat{\co}_{C,p} \rightarrow K[[y]]$ is an isomorphism of formal $K$-algebras. 
\item
$\cf$ is a coherent sheaf of rank $r$ on $C$ such that $H^0(C, \cf )= H^1(C, \cf) =0$;
\item
$\hat{\varphi} :\hat{\cf}_p \rightarrow \hat{\co}_{C,p}^{\oplus r}$ is an isomorphism of sheaves of $\hat{\co}_{C,p}$-modules.
\end{enumerate}
\end{Def}

\begin{rem}
\label{R:Q.2.9}
The datum from definition \ref{Q2.9} can be transferred to a datum from definition \ref{D:proj_spec_data} as follows (cf. \cite[Rem. 2.13]{Q}): put $y=z^r$, then $\pi$ can be obtained from $\rho$ via the composition
$$
\hat{\co}_{C,p}\xrightarrow{\rho}K[[z^r]]\hookrightarrow K[[z]].
$$ 
To construct $\hat{\phi}$ from $\hat{\varphi}$ let's fix an isomorphism of  $\hat{\co}_{C,p}$-modules 
$$
\hat{\psi} : \hat{\co}_{C,p}^{\oplus r} \simeq K[[z]]\cdot z\quad (\alpha_0, \ldots ,\alpha_{r-1})\mapsto z^r\pi (\alpha_0)+z\pi (\alpha_1)+\ldots +z^{r-1}\pi (\alpha_{r-1}).
$$ 
Then we can put $\hat{\phi}:=\hat{\psi}\circ \hat{\varphi}$. 

The data from definition \ref{Q2.9} are more convenient for tracking what happens when choosing different trivializations.
\end{rem}

\subsection{Schur pairs and embedded Schur pairs}
\label{S:embedded}

\begin{Def}
\label{D:Q5.1} 
Let $A$ be a $K$-subalgebra of $K((z))$, and $r\in \dn$. $A$ is said to be an algebra of rank $r$ if 
$r= \gcd (\ord(a)|\quad  a\in A)$, where the order is defined in the same way as the usual order  in $D_1$.
\end{Def}

\begin{lemma}{(\cite{Mu} or \cite[L.5.2. and Rem]{Q}, \cite[Rem 9.13]{Zheglov_book})}
\label{Q5.2}
$A\subset K((z))$ is an $K$-subalgebra of rank $r$ if and only if there is a monic element $y\in K[[z]]$ of order $-r$ such that 
\begin{itemize}
\item
$A\subset K((y))$,
\item 
$K((y))/(A+K[[y]])$ is a finitely generated $K$-vector space.
\end{itemize}
\end{lemma}

\begin{Def}
\label{D:support} 
Let $W$ be a $K$-subspace in $K((z))$. The {\it support} of an element $w\in W$ is its highest symbol, i.e. $\sup (w):=HT(w)z^{-\ord (w)}$. The {\it support} of the space is\\
 $\Sup W:=\langle \sup (w)|\quad w\in W\rangle$. 
\end{Def}

\begin{Def}
\label{D:Schur_pair}
An embedded Schur pair of rank $r$  is a pair $(A,W)$ consisting of 
\begin{itemize}
\item
$A\subset K((z))$ a $K$-subalgebra of rank $r$ satisfying $A\cap K[[z]]=K$;
\item 
$W\subset K((z))$ a $K$-subspace with $\Sup W= K[z^{-1}]$
\end{itemize}
such that $W\cdot A\subseteq W$.
\end{Def} 

\begin{rem}
\label{R:Quandt}
It is possible to extend the notion of a Schur pair and define a {\it category} $\cs_r$ of Schur pairs, whose objects are Schur pairs of rank $r$, see \cite{Mu}. In loc. cit. the embedded Schur pairs were called as just Schur pairs; our terminology comes from \cite{Q}, where the relative version of the classification theory was given. In \cite{Q} another version of the Schur pair was used.

Namely, by a Schur pair of rank $r$  there the author meant a pair $(A ,W )$ consisting of elements $A\in K((y))$, and $W\in K((y))^{\oplus r}$ such that 
\begin{itemize}
\item 
$A$ is a $K$-subalgebra of $K((y))$ and $A\cap K[[y]]=K$,
\item
the natural action of $K((y))$ on $K((y))^{\oplus r}$ induces an action of $A$ on $W$ s.t. $A\cdot W\subseteq W$, and $W\cap K[[y]]^{\oplus r}=0$, $K((y))^{\oplus r}/(W+K[[y]]^{\oplus r})=0$. 
\end{itemize}
 
\end{rem}

The connection between Schur pairs and embedded Schur pairs is the following:

\begin{Prop}{(\cite[Prop.5.4]{Q})}
\label{Q5.4}
There is a canonical one-to-one correspondence between Schur pairs of rank $r$ and embedded Schur pairs $(A,W)$ of rank $r$ with the extra condition that
$$
A\subset K((z^r)).
$$
\end{Prop}

For reader's convenience we copy the proof here, because we use it in our explicit constructions:
\begin{proof}
Let us start with an embedded Schur pair $(A,W)$ of rank $r$ such that $A\subset K((z^r))$. Set $y:=z^r$. Then, by lemma  \ref{Q5.2}, $A\subset K((y))$ is a subalgebra of rank one.  Now we identify:
$$
K[[z]]\cdot z=\bigoplus_{i=1}^r K[[z^r]]\cdot z^i= K[[y]]^{\oplus r}.
$$ 
This identification extends to an isomorphism of $K((z))$ with $K((y))^{\oplus r}$ and so we end up with $W\subset K((y))^{\oplus r}$. Since $K[[z]]\cdot z$ translates into $K[[y]]^{\oplus r}$, $W$ is a subspace  such that $W\cap K[[y]]^{\oplus r}=0$, $K((y))^{\oplus r}/(W+K[[y]]^{\oplus r})=0$.

That also clarifies the inverse construction. We formally set $y:=z^r$ and translate the data back using lemma \ref{Q5.2}. 
\end{proof}

\begin{Def}
\label{D:equiv_Schue_pair}
Two embedded Schur pairs   $(A_i, W_i)$, $i=1,2$ of rank $r$ are {\it equivalent} if there exists an admissible operator $T$ such that $A_1=T^{-1}A_2T$, $W_1=W_2\cdot T$. An operator $T=t_0+t_1\partial^{-1}+\ldots $ is called {\it admissible} if $T^{-1}\partial T\in K((\partial^{-1}))$. 
\end{Def}

\begin{theorem}
\label{T:classif_Schur_pairs}
There are one-to-one correspondences  $\mbox{$[B]$ of rank $r$}\longleftrightarrow \mbox{$[(A,W)]$ of rank $r$}$ and $\mbox{$[B]$ of rank $r$}/\thicksim\longleftrightarrow \mbox{$[(A,W)]$ of rank $r$}/\thicksim$, where $(A,W)$ means the embedded Schur pair, and $\thicksim$ means a scale automorphism $x\mapsto c^{-1}x$, $\partial\mapsto c\partial$. 
\end{theorem}

A self contained proof of this theorem see e.g. in \cite[10.3]{Zheglov_book}. The proof is base on the following elements of the Schur-Sato theory.

Let $B\subset D_1$ be an elliptic subring. Then by Schur theory from \cite{Schur}, cf. \cite[T. 4.6, C. 4.7]{Zheglov_book}, there exists an invertible operator $S=s_0+s_1\partial^{-1} +\ldots $ in the usual (Schur's) ring of pseudo-differential operators $E=K[[x]]((\partial^{-1}))$ such that $A:=S^{-1}BS \subset K((\partial^{-1}))$. Consider the homomorphism (of vector spaces)
\begin{equation}
\label{E:Sato_hom}
E\rightarrow E/xE \simeq K((\partial^{-1}))
\end{equation}
(sometimes it is called {\it the Sato homomorphism}). It defines a structure of an $E$-module on the space  $K((\partial^{-1}))$: for any $P\in K((\partial^{-1}))$, $Q\in E$ we put  $P\cdot Q= PQ$ (mod $xE$). 

Analogously, the homomorphism 
$$
1\circ :\tilde{K}[A_1]((\tilde{D}^{-1}))\simeq \tilde{K}^{\oplus q}((\tilde{D}^{-1})) \subset E_q\rightarrow  \tilde{K}((\tilde{D}^{-1})), \quad \sum_l p_l\tilde{D}^l \mapsto  \sum_l p_{l,0}\tilde{D}^l
$$
(cf. \cite[Lem 3.3]{GZ24}) defines a structure of a $\tilde{K}[A_1]((\tilde{D}^{-1}))$-module on the space  $\tilde{K}((\tilde{D}^{-1}))$: for any $P\in \tilde{K}((\tilde{D}^{-1}))$ and $Q\in \tilde{K}[A_1]((\tilde{D}^{-1}))$ we put $P\cdot Q= 1\circ (PQ)$. Analogously, $K((\tilde{D}^{-1}))\subset \tilde{K}((\tilde{D}^{-1}))$ is a right $K^{\oplus q}((\tilde{D}^{-1}))$-module. 

Now define the space $W:= F\cdot S \subset K((\partial^{-1}))$, where $F=K[\partial ]$. Note that $W$ is an $A$-module, where the module structure is defined via the multiplication in the {\it field} $K((\partial^{-1}))$ and this module structure is induced by the $E$-module structure on $K((\partial^{-1}))$, because $K((\partial^{-1}))\subset E$ and $W\cdot A=(F\cdot S)\cdot (S^{-1}BS)=F\cdot (BS)=(F\cdot B)\cdot S= F\cdot S= W$. 
Note also that the modules $W$ and $F$ are isomorphic ($W$ is an $A$-module, $F$ as a $B$-module, and clearly $A\simeq B$). For convenience of notation, we will replace $\partial^{-1}$ by $z$ in the field $K((\partial^{-1}))$, i.e. $A,W\subset K((z))\simeq K((\partial^{-1}))$. The ring $B$ can be reconstructed from $(A,W)$ with the help of the Sato operator, cf. \cite{Mu}. 

Analogously, we can define the space $W':=F'\cdot S\subset \tilde{K}((\tilde{D}^{-1}))$, where $F'=\tilde{K}[\tilde{D}]$ and $S\in \tilde{K}[A_1]((\tilde{D}^{-1}))$ is an operator from lemma \cite[Lem. 3.3]{GZ24} (cf. lemma \ref{L:conjugation in E_q} below) or $W'':=F''\cdot S\subset K((\tilde{D}^{-1}))$, where $F''=K[\tilde{D}]$ if $S\in K^{\oplus q}((\tilde{D}^{-1}))$.

\subsection{Classification theorems}
\label{S:classif_thms}

The classification of commutative subrings in terms of algebraic-geometric spectral data is based on the construction of the Krichever map. 

Recall that the Krichever map 
$$
\chi_0:(C,p,\cf ,\pi , \hat{\phi} ) \rightarrow (A,W) \quad \chi_0:(C,p,\cf ,\rho , \hat{\varphi} ) \rightarrow (A,W)
$$
is defined for data from definition \ref{D:proj_spec_data} (for details see \cite[Ch.10]{Zheglov_book}) or for data from definition \ref{Q2.9} (and $(A,W)$ is an embedded Schur pair in the first case and a Schur pair from remark \ref{R:Q.2.9} in the second case). 

For data from definition \ref{D:proj_spec_data} it is defined as the embeddings 
$H^0(C\backslash p, \co_C )\simeq \limind_{n\ge 0} H^0(C, \co_C (np)) \hookrightarrow K((z)) \mbox{,}$
$H^0(C\backslash p, \cf )\simeq \limind_{n\ge 0} H^0(C, \cf (np)) \hookrightarrow K((z)) \mbox{,}$ 
through the natural maps
$$
\alpha_n: H^0(C, \cf (np))\hookrightarrow {\cf (np)}_p\simeq f^{-n} ({\cf }_p) \hookrightarrow K((z)) 
$$
defined for any torsion free sheaf $\cf$ and $n\ge 0$, where the last embedding is the embedding $f^{-n}{\cf }_p \stackrel{\hat{\phi} }{\hookrightarrow } 
\pi (f)^{-n} K[[z]]\cdot z {\hookrightarrow} K((z))$, and $f$ is a local generator of the ideal $\idm_p$.

For data from definition \ref{Q2.9} it is defined similarly: 
$$
A = \rho (H^0(C\backslash p,\co_C))\subset K((y)),
$$
$$
W=\rho \circ \hat{\varphi} (H^0(C\backslash p,\cf ))\subset K((y))^{\oplus r}. 
$$

The inverse geometric constructions (from Schur pairs to geometric data) is described in \cite[\S 9]{Zheglov_book}. Recall the construction of the curve and the sheaf from the embedded Schur pair $(A,W)$ of rank $r$: set $\tilde{A}=\oplus_{i\ge 0}A_is^i$, where $A_i:=\{a\in A|\quad \ord (a)\le ir\}$, and set 
$\tilde{W}=\oplus_{i\ge 0}W_{ir-1}s^i$, where $W_i=\{w\in W| \quad \ord (w)\le i\}$. Then
$$
C=\Projj \tilde{A}, \quad \cf =\Projj \tilde{W},
$$
and the point $p$ is given by the homogeneous ideal $(s)$ in $\tilde{A}$.

Using theorem \ref{T:classif_Schur_pairs} and inverse geometric constructions the following classification theorem can be proved:
\begin{theorem}{(\cite[Th. 10.26]{Zheglov_book})}
\label{T:classif2}
There are  one-to-one correspondences 
$$
[B\subset {D_1}\mbox{\quad of rank $r$}]  \longleftrightarrow  [(C,p,\cf , \pi, \hat{\phi} )\mbox{\quad of rank $r$}]/\simeq \longleftrightarrow  [(C,p,\cf , \rho, \hat{\varphi} )\mbox{\quad of rank $r$}]/\simeq
$$
where 
\begin{itemize}
\item
$[B]$ means a class of equivalent commutative elliptic subrings (i.e. $B$ containing a monic differential operator),  where  $B\sim B'$ iff $B=f^{-1}B'f$, $f\in D_1^*=K[[x]]^*$. 
\item $C$ is a (irreducible and reduced) projective curve over $K$, $p$ is a regular $K$-point and $\cf$ is a torsion free sheaf of rank $r$ with $H^0(C, \cf)=H^1(C,\cf )=0$ (a spectral sheaf). 
\item $\simeq$ mean isomorphisms of data. 
\end{itemize}
\end{theorem}

\begin{rem}
\label{R:proof_explanations}
The self contained proof of the first row of correspondences is given in \cite[Ch. 9, 10]{Zheglov_book}, this proof is an elaborated version of Mulase's proof from \cite{Mu} in a spirit of works \cite{Parshin2001}, \cite{Os}  and their  higher dimensional generalisations in \cite{Zheglov2013}, \cite{KOZ2014} (and the correspondence between data is described in sections \ref{S:geom_data}, \ref{S:embedded}).  

It follows directly from the definition of data isomorphism \ref{D:proj_spec_data_isom} that, taking any change of trivialisation $\pi$ by a composition $\bar{h}\circ \pi$, where $\bar{h}$ is an automorphism from definition \ref{D:proj_spec_data_isom}, we can find an appropriate trivialisation $\hat{\phi}'$ such that the data $(C,p,\cf ,\pi , \hat{\phi} )$ and $(C,p,\cf ,\bar{h}\circ \pi , \hat{\phi}' )$ are isomorphic. 

Just the same proof shows that, if we relax definition of data isomorphism by allowing arbitrary continuous automorphisms $\bar{h}: K[[z]]\rightarrow K[[z]]$, then we get a one-to-one correspondence with the equivalence classes of rings with respect to the weakened equivalence relation: 
\end{rem}

\begin{theorem}
\label{T:classif3}
There are  one-to-one correspondences
$$
[B\subset {D_1}\mbox{\quad of rank $r$}]/\thicksim  \longleftrightarrow  [(C,p,\cf , \pi , \hat{\phi} )\mbox{\quad of rank $r$}]/\cong \longleftrightarrow  [(C,p,\cf , \rho , \hat{\varphi} )\mbox{\quad of rank $r$}]/\cong ,
$$
where 
\begin{itemize}
\item
$\thicksim$ means "up to a scale automorphism $x\mapsto c^{-1}x$, $\partial\mapsto c\partial$".
\item $\cong$ means the modified isomorphisms of data. 
\end{itemize}
\end{theorem}

\section{Explicit parametrisation of the moduli space of rank $r$ spectral sheaves}
\label{S:moduli}

In this section we prove theorem \ref{T:parametrisation} thus 
giving an explicit description of the moduli space of spectral sheaves of arbitrary rank on a spectral curve. Suppose $\tilde{K}$ is an algebraic closure of $K$. 

Assume $P,Q\in D_1$ are differential operators of positive order, where $Q$ is normalized and $P$ is monic and $[P,Q]=0$. We'll call a normal form $P'$ of the operator $P$ with respect to $Q$ as {\it partially normalized}, if it satisfies conditions of the following lemma.

\begin{lemma}
\label{L:normalisation}
Let $Q,P$ be differential operators as above. Assume $\ord (Q)=\Ord (Q)=q=dm$, $\ord (P)=\Ord (P)=p=dn$, 
$GCD(n,m)=1$. 
For $i=1, \ldots , q-1$, $\alpha = 0, \ldots , d-1$  define the sets $N_{i,\alpha}:=\{j\in \dn |\quad \mbox{$\alpha +d$($jn$ mod $m$)$\ge \max\{i,d\}$}\}$. 
  
Then there is a partially normalised normal form $P'$ of $P$ with respect to $Q$ which is described via 
its image under the embedding $\hat{\Phi}$ as follows: 
$$
\hat{\Phi}(P')=\tilde{D}^p +\sum_{l=-q+1}^{p-1} p_l \tilde{D}^l, \quad p_l=(p_{l,0},\ldots ,p_{l,q-1})\in \tilde{K}^{\oplus q},
$$
where $p_{l,j}=0$ for indices $j$ satisfying the following rules (all indices are taken mod $q$): for $\alpha = 0, \ldots , d-1$  
$$
\begin{cases}
\mbox{If $p\ge q$ then} 
\begin{cases}
p_{l,\alpha +(j-1)p}=\mbox{$0$ for $j\in N_{p-l, \alpha}$}\quad l>p-q \\
p_{l,j}=\mbox{$0$ for $j=0,\ldots ,-l-1$} \quad l<0
\end{cases}
\\
\mbox{If $p< q$ then}
\begin{cases}
p_{l,\alpha +(j-1)p}=\mbox{$0$ for $j\in N_{p-l, \alpha}$}\quad l\ge 0 \\
p_{l,j}=\mbox{$0$ for  $j=0,\ldots ,-l-1$ or if $j=\alpha +(k-1)p$ and $k\in N_{p-l, \alpha}$}\quad p-q<l<0 \\
p_{l,j}=\mbox{$0$ for $j=0,\ldots ,-l-1$} \quad l\le p-q
\end{cases}
\end{cases}
$$
Let's call coefficients $p_{l,j}$ supplementary to the list above as {\it coordinates} of  $P'$. If $P'$ and $P''$ are two operators with different values of coordinates, then there are no invertible monic  operators of order zero $S\in C(\partial^q)\subset \hat{D}_1^{sym}$, $S=1+\sum_{l=1}^{q-1}s_l\int^l$, with additional condition $\Phi (s_{l})_{\alpha}=0$ for $\alpha = 0, \ldots , d-1$ and any $l>0$, such that $P'=S^{-1}P''S$, and there are no invertible $1\neq S\in C(\partial^q)$ with this additional condition such that $P'=S^{-1}P'S$.
\end{lemma}

\begin{rem}
\label{R:normalisation}
This lemma is analogous to \cite[Lem. 3.4]{GZ24}, but the conclusion is weaker. Namely, the form $P'$ may have a non-trivial stabilizer from $C(\partial^d)$, and for two operators $P',P''$ with different values of coordinates can exist an invertible operator $S\in C(\partial^q)$  such that $P'=S^{-1}P''S$. The examples of such operators will appear in section \ref{S:example}. 

A partially normalized normal form $P'$ from lemma is obtained from $P$ by conjugation to the Schur operator $\hat{S}\in \hat{D}_1^{sym}$ for the operator $Q$, i.e., by definition, a monic operator of order zero such that $\hat{S}Q\hat{S}^{-1}=\partial^q$, $\hat{S}P\hat{S}^{-1}=P'$. From the proof below we'll see that such operator is uniquely defined (in particular, the partially normalized normal form is uniquely defined) if we pose an extra condition on its homogeneous components $\hat{S}_{-l}$, $l=1,\ldots ,d-1$. Recall that a Schur operator satisfies condition $A_q(0)$ (\cite[Prop. 2.5, Cor. 2.3]{GZ24}) and it  is defined up to multiplication by a monic invertible zeroth order operator $S\in C(\partial^q)$. Then we can choose any condition that uniquely distinguishes this operator in the coset defined by multiplication by  operators $S$ with {\it non-zero} coefficients  $\Phi (s_{l})_{\alpha}=0$ for some $l$ and some $\alpha \in \{0, \ldots , d-1\}$. It is possible due to a special G-form of homogeneous components of operators satisfying condition $A_q(0)$ (see the list of notations).
\end{rem}

\begin{proof}
The proof is similar to the proof of \cite[Lem. 3.4]{GZ24}. For convenience of the reader we give it here.

To show that such normalised form exist, we can follow the arguments in \cite[Lem. 3.3]{GZ24}. If $P'$ is a given normal form of $P$ with respect to $Q$, obtained by conjugation by a Schur operator $\hat{S}$ satisfying an additional condition from remark \ref{R:normalisation}, we will find a monic operator of order zero $S\in C(\partial^q)$ with additional condition $\Phi (s_{l})_{\alpha}=0$ for $\alpha = 0, \ldots , d-1$ and any $l>0$ (thus it will be automatically invertible) step by step, as a product of operators $S_l=1+s_l \int^{l}\in C(\partial^q)$, $l=1,\ldots ,q-1$. Since $\hat{\Phi}$ is an embedding of rings, we can provide all calculations in the ring $E_q$. Note that by definition of $\hat{\Phi}$ and from \cite[Lem. 2.4]{GZ24} we get $\hat{\Phi}(S_l)=1+\tilde{s}_l \tilde{D}^{-l}$, where $\tilde{s}_{l,j}=0$ for $j=0, \ldots , l-1$. Using calculations from \cite[Lem. 3.3]{GZ24}, we get for 
$\tilde{P}'':=\hat{\Phi}(S_l^{-1}P'S_l)= \tilde{D}^p+\ldots + \tilde{p}_l\tilde{D}^{p-l}+\ldots$ 
the following $d$  systems of linear equations: for $\alpha =0, \ldots ,d-1$ 
$$
\begin{cases}
p_{l,\alpha }-\tilde{s}_{l,\alpha}+\tilde{s}_{l,\alpha +p}=\tilde{p}_{l,\alpha} \\
p_{l,\alpha +p}-\tilde{s}_{l,\alpha +p}+\tilde{s}_{l,\alpha +2p}=\tilde{p}_{l,\alpha +p} \\
\ldots \\
p_{l,\alpha +(m-1)p}-\tilde{s}_{l,\alpha +(m-1)p}+\tilde{s}_{l,\alpha}=\tilde{p}_{l,\alpha+(m-1)p}
\end{cases}
$$ 
because $GCD(p,q)=d$ (here again all indices are taken modulo $q$). Now we can put $\tilde{s}_{l,\alpha}=0$ for all $l=1,\ldots ,q-1$, $\alpha =0, \ldots ,d-1$, and, starting with the first equation, we can see that for $j$-th equation in each system of linear equations, where $j\in N_{l,\alpha }$, we can find $\tilde{s}_{l,\alpha + jp}$ such that $\tilde{p}_{l,\alpha + (j-1)p}=0$. On the other hand, $\tilde{s}_{l,\alpha + jp}=0$ for all $j\notin N_{l,\alpha }$. Thus, we uniquely determine $\tilde{s}_l$ such that conditions of lemma satisfied for $(p-l)$-th homogeneous component of $\tilde{P}''$ and $\tilde{s}_{l,\alpha}=0$ for all $\alpha =0, \ldots ,d-1$. 

On the other hand, since $\tilde{P}''\in C(\partial^q)$, we know from lemma \cite[Lem. 2.4]{GZ24} that  $\tilde{p}_{l,j}=0$ for $j=0,\ldots , -l-1$ if $l<0$ (in particular, the number of zero coefficients for $p-q<l<0$ in case $p<q$ is constant). Taking $S=\prod_{j=1}^{q-1}S_j$ we get the needed operator: the normal form $S^{-1}P'S$ will satisfy all conditions as stated. Without loss of generality we can assume that the operator $S\hat{S}$ satisfies an additional condition from remark \ref{R:normalisation}.

To prove the second statement (the uniqueness of such partially normalised normal form up to conjugation with a monic invertible operator of order zero with additional conditions), just note that the same equations as above show that all other homogeneous components of $S$ must be zero (as all coefficients of  $P'$ and $P''$ from the list in the formulation are zero) and therefore the equality $P'=S^{-1}P''S$ is impossible. 

Finally, note that if $\hat{S}'$ is another operator satisfying an additional condition from remark \ref{R:normalisation} such that $P''=(\hat{S}')^{-1}P\hat{S}'$ is a partially normalized normal form, then
$S':=S\hat{S}(\hat{S}')^{-1}\in C(\partial^q)$ must be a monic invertible operator  with {\it non-zero} coefficients  $\Phi (s_{l}')_{\alpha}=0$ for some $l$ and some $\alpha \in \{0, \ldots , d-1\}$ - a contradiction with the additional condition from remark \ref{R:normalisation}. Thus, the partially normalized Schur operator $\hat{S}$ and the corresponding partially normalized normal form $P'$ are uniquely determined. 

\end{proof}

Before we prove of our main theorem recall some necessary facts about the Burchnall-Chaundy polynomial of a pair of commuting differential operators. 

Recall the famous Burchnall-Chaundy lemma (\cite{BC1}) says  that any two commuting differential operators $P,Q\in D_1:=K[[x]][\partial ]$ are algebraically dependent. More precisely, if the orders $n, m$ of operators $P,Q$ are coprime (i.e. the rank of the ring $K[P,Q]$ is 1), then there exists an irreducible polynomial $f(X,Y)$ of weighted degree $v_{n,m}(f)=mn$ of {\it special form} (here the weighted degree is defined as in Dixmier's paper \cite{Dixmier}): $f(X,Y) = \alpha X^m\pm Y^n+\ldots $ (here $\ldots$ mean terms of lower weighted degree, $0\neq\alpha\in K$; in particular, for coprime $m$ and $n$ the polynomial $f$ is automatically irreducible), such that $f(P,Q)=0$. A similar result for commuting operators of rank $r$ was established in \cite{Wilson}, in this case $m=\ord (Q)/r$, $n=\ord (P)/r$, and again $GCD(m,n)=1$. More precisely, by theorem in {\it loc. cit.}, (cf. also \cite{Previato2019}) the polynomial given by the differential resultant $f(X,Y)=\partial Res(P-X,Q-Y)$ gives an algebraic relation between $P$ and $Q$, and $f=h^r$, where $h$ is irreducible and $r=\rk K[P,Q]$. In particular, if $r=1$, then $f$ is irreducible; besides, as it follows from definition of differential resultant, $f(X,Y) = (\alpha X^{\ord (Q)/d}\pm Y^{\ord (P)/d})^d+\ldots $, where $d=GCD (\ord (Q), \ord (P))$, even if the orders of $P,Q$ are not coprime. 

Recall one lemma from \cite{GZ24}:
\begin{lemma}{(\cite[Lem. 3.3]{GZ24})}
\label{L:conjugation in E_q}
Assume $\tilde{P}\in \tilde{K}^{\oplus k}((\tilde{D}^{-1}))\subset E_k$ is a monic operator with $\ord_{\tilde{D}}(\tilde{P})=p$. Let $p=dn$, $k=q=dm$, where $d=GCD (p,q)$. 

Then there exists a monic invertible operator $\tilde{S}\in \tilde{K}^{\oplus k}((\tilde{D}^{-1}))\subset E_k$ with $\ord_{\tilde{D}}(\tilde{S})=0$ such that all coefficients of the operator $\tilde{S}^{-1}\tilde{P}\tilde{S}$ commute with $\tilde{D}^d$. Moreover, there exists a unique such operator with additional condition on its coefficients $\tilde{s}_i$: for all $i$ and for $\alpha =0, \ldots , d-1$ $\tilde{s}_{i,\alpha}=0$. 
\end{lemma}

The last assertion of this lemma is proved in \cite[Th. 5.2]{Zheglov25} for $d=1$, for generic $d$ the proof is almost the same. 

The following lemma is a generalisation of \cite[Lem. 3.5]{GZ24}.
\begin{lemma}
\label{L:solution_of_BC_equation}
Let $f(X,Y) =  X^q\pm Y^p+\ldots \in K[X,Y]$ be a Burchnall-Chaundy polynomial with coprime $p,q$. Assume  $P'\in C(\partial^{dq})\subset \hat{D}_1^{sym}\hat{\otimes}\tilde{K}$, $d\in \dn$ is a monic operator with $\Ord (P')=dp$ such that $f(P',\partial^{dq})=0$. 

Then $a:= S^{-1}\hat{\Phi}(P')S\in \tilde{K}((\tilde{D}^{-1}))$, where $S$ is an operator from lemma \ref{L:conjugation in E_q}, and $a$ is the uniquely defined monic element in $\tilde{K}((\tilde{D}^{-1}))$, satisfying the equation $f(a,\tilde{D}^q)=0$. 
\end{lemma}

\begin{proof}
From lemma \ref{L:conjugation in E_q} we know that all coefficients of $a$ commute with $\tilde{D}^d$. Since $p,q$ are coprime and $\Ord (a)=dp$, the standard arguments from Hensel's lemma show that all such coefficients must be constant (i.e. all their components in $\tilde{K}^{\oplus dq}$ are equal). By the same reason, such a monic series $a$ with constant coefficients is uniquely determined by the equation $f$. 
\end{proof}

\begin{rem}
\label{R:solution_of_BC_equation}
For generic BC-polynomial $f(X,Y) = (\alpha X^{\ord (Q)/d}\pm Y^{\ord (P)/d})^d+\ldots $ the conclusion of lemma may be not true. This explains the special condition on orders of operators in definition \ref{D:normalised_normal_form}, see also proof of theorem \ref{T:parametrisation} below.
\end{rem}

For the proof of our main theorem we'll need also the following lemma.
\begin{lemma}
\label{L:T=SS}
Let $\tilde{S}\in \tilde{K}^{\oplus k}((\tilde{D}^{-1}))\subset E_k$ be the operator from lemma \ref{L:conjugation in E_q} satisfying additional conditions on its coefficients $\tilde{s}_i$: 
$\tilde{s}_{i,0}=0$. 

Then there is a unique monic zeroth order operator $T$ with all its homogeneous components don't depending on $A_i$ and $B_j$ (i.e. $T\in \tilde{K}[[x]]((\partial^{-1}))\subset E_k$, cf. footnote 13 in \cite{Zheglov25}) such that $\tilde{S}T$ is a monic zeroth order operator satisfying condition $A_k(0)$\footnote{Here we adopt definition of condition $A_k(0)$ to the ring $E_k$ in the following sense: condition 2 from definition (totally freeness) is replaced by an equivalent system of equations \cite[(2.3)]{Zheglov25} on coefficients, all other conditions are the same, see footnote 13 in loc. cit.}. 
\end{lemma}

This lemma is proved in item 4 of the proof of \cite[Th. 5.2]{Zheglov25} for $d=1$, for generic $d$ the proof is just the same.

\subsection{The proof of theorem \ref{T:parametrisation}}

Without loss of generality we can work over the algebraically closed field $\tilde{K}$.\footnote{All correspondences in the classification theorems \ref{T:classif2}, \ref{T:classif3}, \ref{T:classif_Schur_pairs} are relevant to the extension of scalars, i.e. if the sheaf $\cf$ with vanishing cohomologies is defined over an extension $\bar{K}$ of $K$, then the corresponding Schur pairs and the ring of commuting differential operators are determined over $\bar{K}$, and vice versa. As we'll see below, the corresponding partially normalised normal forms will be determined over $\bar{K}$ as well. This follows from generic theory of normal forms from \cite{GZ24}.} Let $\cf$ be a  torsion free rank $r$ sheaf on $C$ with vanishing cohomologies. Take $\pi$ as in the formulation, and choose some trivialisation 
$\hat{\phi}$ of the sheaf $\cf$ at the point $p$. By theorem \ref{T:classif3} this data corresponds to a  uniquely defined (up to a scale transform) {\it normalised} commutative elliptic subring $B\subset D_1$ of rank $r$, such that the generator $w_1$ corresponds to a normalised operator $P_1\in B$ with $\ord (P_1)=rq=r\ord (w_1)$ via the isomorphism $\co_C(C_0)\simeq B$ constructed in the correspondence of data in loc. cit. Moreover, the ring $B$ and the operator $P_1$ don't depend on change of trivialisation $\pi$ in the isomorphism class of this data (cf. remark \ref{R:proof_explanations} and  proof in \cite[Ch. 9,10]{Zheglov_book}), so that we can assume without loss of generality that $\pi$ is chosen so that $w_1$ corresponds to the monomial $z^{-rq}\in A$ in the corresponding embedded Schur pair $(A,W)$, and that this pair satisfies condition from proposition \ref{Q5.4}, i.e. $A\subset \tilde{K}((z^r))$. Besides, the generators $w_i$ corresponds to  elliptic differential operators $P_1, \ldots ,P_m$ with $\ord (P_i)=r\ord (w_i)$. 

By  \cite[Prop. 2.3]{GZ24}  and lemma \ref{L:normalisation} there exists a monic zeroth order  Schur operator $S\in \hat{D}_1^{sym}$ such that $B':=S^{-1}BS\in C(\partial^{rq})$ and $P_2'=S^{-1}P_2S$ is a partially normalised normal form of $P_2$ with respect to $P_1$, i.e. $B'$ is a partially normalised normal form w.r.t. $P_1'=\partial^{rq}, \ldots ,P_m'$ which have the same orders as $P_i$. Note that the scale transform is compatible with conjugation by $S$ and that the coefficients of $\hat{\Phi}(P_i')$ are invariant under any scale transform of $P_i'$ for all $i$. Note also that $A\simeq A'=\tilde{S}^{-1}\hat{\Phi}(B')\tilde{S}\subset \tilde{K}((\tilde{D}^{-r}))$ (as rings) and $W\simeq W'=F'\cdot \tilde{S}$ (as modules over $A$ ($A'$)), where $\tilde{S}$ is the operator from lemma \ref{L:conjugation in E_q} constructed for the operator $P_2'$, and the isomorphism is given via $z\mapsto \tilde{D}^{-1}$. Indeed, since $GCD(\ord (P_1),\ord (P_i))=r=\rk B$ for any $i>1$, we have by Wilson's theorem that the BC-polynomial for any pair $P_1,P_i$ satisfies condition of lemma \ref{L:solution_of_BC_equation}. So, by this lemma $\tilde{S}^{-1}\hat{\Phi}(P_2')\tilde{S}\in \tilde{K}((\tilde{D}^{-r}))$. Then it's easy to see that coefficients of all operators $\tilde{S}^{-1}\hat{\Phi}(P_i')\tilde{S}$, $i>2$ commute with $\tilde{D}^r$, and the pair $\partial^{rq}, \tilde{S}^{-1}\hat{\Phi}(P_i')\tilde{S}$ gives a solution to the BC-polynomial for the pair $P_1,P_i$. Then by the arguments from the proof of this lemma we get again $\tilde{S}^{-1}\hat{\Phi}(P_i')\tilde{S}\in \tilde{K}((\tilde{D}^{-r}))$. The isomorphism $W\simeq W'$ is obvious.

So, we get that the graded $\tilde{B}$-module $\tilde{F}$, $\tilde{A}$-module $\tilde{W}$, $\tilde{A'}$-module $\tilde{W'}$ and $\tilde{B'}$-module $\tilde{F'}$ are isomorphic (cf. section \ref{S:embedded}). Since the sheaf $\cf$ is determined up to an isomorphism by the graded $\tilde{A}$-module $\tilde{W}$, it can be recovered by $B'$. The form $B'$ is not defined uniquely, because the Schur operator $S$ is defined up to multiplication by an invertible operator from $C(\partial^r)\subset C(\partial^{rq})$. Namely, for any $S_0\in C(\partial^r)\subset C(\partial^{rq})$ such that $S_0^{-1}P_2'S_0$ is partially normalised (so that $S_0^{-1}B'S_0$ is another partially normalised normal form with respect to $\partial^{rq},S_0^{-1}P_2'S_0, \ldots S_0^{-1}P_m'S_0$) and such that the graded $\tilde{B'}$-module $\tilde{F'}$ is isomorphic to the graded $S_0^{-1}\tilde{B'}S_0$-module $\tilde{F'}$, the form $S_0^{-1}B'S_0$ will determine the same sheaf $\cf$. Note that the last condition on the isomorphism of graded modules implies the restriction on the order of $S_0$: it must be less that $r$ (otherwise these modules will not be isomorphic as graded modules). This explains why the equivalence class of the point $[B']$ from $X_{[\co_C(C_0)]}$ corresponds to the isomorphism class of $\cf$.

However, we need yet to trace out what happens if we choose another trivialisation of $\cf$ at the point $p$. We need to check that the corresponding partially normalised normal form $B''$ provides the same equivalence class of the point determined by $B'$. This can be done via Schur pairs from remark \ref{R:Quandt} corresponding to our embedded Schur pairs with fixed $\pi$. Any change of trivialisation is given by an invertible matrix $M$ from $GL_r(\tilde{K}[[y]])$, so that after change the Schur pair $(A,W)$ will be replaced by a pair $(A,W')$, where $W'=M\circ W=M\circ\rho \circ \hat{\varphi} (H^0(C\backslash p,\cf ))\subset K((y))^{\oplus r}$, and $M$ acts on $W\subset K((y))^{\oplus r}$ via right multiplication: each element from $W$ is raw  $w\in K((y))^{\oplus r}$, it maps to $w':=w\cdot M$. 
Define the matrix $\tilde{M}:=M|_{y\mapsto \tilde{D}^{-r}}\in M_r(\tilde{K}[[\tilde{D}^{-r}]])$. Note that $M$ (and $\tilde{M}$) can be decomposed as $M=M_0M_1$, where $M_0\in GL_r(\tilde{K})$, $M_1=E+M_1'$, $M_1'\in M_r(\tilde{K}[[y]]\cdot y)$ (correspondingly, $\tilde{M}=\tilde{M}_0\tilde{M}_1$ with similar properties). Next, note that $\tilde{M}_0=\psi \circ \hat{\Phi} (S')$, where $S'\in C(\partial^r)\subset \hat{D}_1^{sym}$ is an invertible operator, and $\tilde{M}_1=\psi (\tilde{S})$, where 
$\tilde{S}\in \tilde{K}^{\oplus r}[[\tilde{D}^{-1}]]$ is invertible monic zeroth order operator. Denote $\tilde{S}':= \hat{\Phi} (S')\in \tilde{K}^{\oplus r}[\tilde{D}^{-1},\tilde{D}]$, $\tilde{S}^f:=\tilde{S}'\tilde{S}\in \tilde{K}^{\oplus r}((\tilde{D}^{-1}))$. Let $\tilde{S}^f=\sum_i\tilde{s}_i^f\tilde{D}^{-i}$. 
Define the operator $\tilde{S}^{\oplus q}=\sum_i\tilde{s}_i\tilde{D}^{-i} \in \tilde{K}^{\oplus rq}((\tilde{D}^{-1}))$ by the rule $\tilde{s}_i=(\tilde{s}_i^f,\ldots , \tilde{s}_i^f)$ ($q$ times). Thus, $\tilde{S}^{\oplus q}$ commutes with $\tilde{D}^r$ and by construction $\Ord \tilde{S}^{\oplus q}<r$. 

The key observation is that, due to special choice of the isomorphism $\hat{\psi}$ from remark \ref{R:Q.2.9} we have 
$$
\hat{\psi}(w\cdot M) = \hat{\psi}(w)\cdot \tilde{S}^{f}, \mbox{\quad hence \quad} \hat{\psi}(W')=\hat{\psi}(W)\cdot \tilde{S}^{f}
$$
and similarly 
$$
\hat{\psi}(w\cdot M) = \hat{\psi}(w)\cdot \tilde{S}^{\oplus q}, \mbox{\quad hence \quad} \hat{\psi}(W')=\hat{\psi}(W)\cdot \tilde{S}^{\oplus q}
$$  
(here $\hat{\psi}(W)$, $\hat{\psi}(W')$ are the corresponding spaces of the embedded Schur pairs).  

Returning back to the embedded Schur pair $(A', W')$ constructed at the beginning of the proof, we get that a change of trivialisation of the sheaf $\cf$ leads to another Schur pair $(A', W'\cdot \tilde{S}^{\oplus q})$. We know that $B'=\tilde{S}A'\tilde{S}^{-1}$, and also we know that $\tilde{S}^{\oplus q}$ commutes with $\tilde{D}^r$ and $A'\in \tilde{K}((\tilde{D}^{-r}))$. Therefore, $\tilde{S}\tilde{S}^{\oplus q}A'(\tilde{S}^{\oplus q})^{-1}\tilde{S}^{-1}=B'$, thus the corresponding partially normalised normal form is the same.

Vice versa, any equivalence class $[x]$, $x\in X_{[\co_C(C_0)]}$,  determines a partially normalised normalised normal form $B'$ w.r.t. some $P_1'=\partial^{rq}, \ldots ,P_m'$ which have the orders from formulation. As it was noticed above, there is the operator $\tilde{S}$  from lemma \ref{L:conjugation in E_q} constructed for the operator $P_2'$ such that $A'=\tilde{S}^{-1}\hat{\Phi}(B')\tilde{S}\subset \tilde{K}((\tilde{D}^{-r}))$. Note that $W':=F'\cdot \tilde{S}$ has support equal to $F'$, because $\tilde{S}$ is a monic invertible operator of order zero. So, $(W', A')$ form an embedded Schur pair of rank $r$\footnote{Since  $\hat{\Phi}(C(\partial^{rq}))$  consists of Laurent polynomials whose coefficients at negative powers of $\tilde{D}$ are not constants, and since $\tilde{S}$ is monic and $A'\subset K((\tilde{D}^{-r}))$, it follows that $A'\cap \tilde{K}[[\tilde{D}^{-1}]]=\tilde{K}$.}. By theorems \ref{T:classif_Schur_pairs} and \ref{T:classif2} this Schur pair determines a normalised commutative subring $B\subset D_1$ of rank $r$ and a torsion free sheaf $\cf\simeq \Projj \tilde{W'}$ of rank one with vanishing cohomologies. 

As it was noticed above, the modules $\tilde{F}$, $\tilde{F}'$ and $\tilde{W}'$ are isomorphic as graded  modules, and all correspondences are compatible with the scale transform. By this reason the maps $[x=B'], x\in X_{[\co_C(C_0)]} \mapsto \cf$ and $\cf \mapsto [B'], B'\in X_{[\co_C(C_0)]}$ are mutually inverse.

\begin{rem}
\label{R:moduli}
This result indicates (like in the case of rank one sheaves) that the moduli space of {\it spectral sheaves} of rank $r$, i.e. sheaves with vanishing cohomologies, is an affine open subscheme of the compactified moduli space of vector bundles (the moduli space of coherent torsion free sheaves with fixed Hilbert polynomial, cf.  \cite{Rego2}, \cite{HL}). We hope to cover this issue in a subsequent work.

\end{rem}

\section{An Example of parametrisation of rank 2 spectral sheaves over a Weierstrass cubic curve}
\label{S:example}

In this section we illustrate our main theorem and calculate the equivalence classes of points (i.e. conjugacy classes of normal forms in matrix presentation) parametrising all torsion free sheaves of rank 2 with vanishing cohomology groups on a curve of arithmetical genus one.  

Let's recall that first classification of commuting ordinary differential operators of orders 4 and 6 with a smooth spectral curve was given by Kichever and Novikov in \cite{KN}. Later Gr\"unbaum in \cite{Grun} using  (super)computer calculations extended this description thus describing all possible commuting differential operators of orders 4 and 6. Previato and Wilson in \cite{PW} posed a problem of explicit description of the spectral sheaf of the commutative subring of ODOs generated by a pair of commuting operators of orders 4 and 6 via their coefficients, and solved this problem for operators having a smooth spectral curve. They used Gr\"unbaum's notation of operators for calculations. In \cite{BZ} the authors solved this problem in generic case. 

In this section all necessary concepts and notations come from papers  \cite{Grun}, \cite{PW} and  \cite{BZ}. First we recall the description of sheaves from \cite{BZ}.

Suppose $\mathfrak{D}=\mathbf{C}[[x]][\partial]$. 
Now assume $B\subset\mathfrak{D}$ is a genus one and rank 2 commutative sub-algebra, with 
$$
B=\mathbf{C}[L_4,L_6]=\mathbf{C}[x,y]/(y^2-h(x))
$$
where $h(x)=4x^3-g_2x-g_3$ for some parameters $g_2,g_3\in \mathbf{C}$. The operator $L_4$ has the form
$$
L_4=(\partial^2+\frac{C_2(x)}{2})^2+[C_1(x)\partial+\partial C_1(x)]+C_0(x)
$$

Denote $\Delta:=g_2^3-27g_3^2$. The spectral curve $X=\overline{V(y^2-h(x))}$ is singular when $\Delta=0$. If $X$ is singular, in the case of $g_2=g_3=0$ we call it as cuspidal while $g_2\neq 0$ as nodal. If the spectral curve $X$ is singular then $s$ denotes its singular point. And $X\xrightarrow{\iota}X,(\lambda,\mu)\rightarrow(\lambda,-\mu)$ is the canonical involution of $X$ and $p=(0:1:0)$ is the  point at "infinity" of $X$.

Define $\mathsf{Sem}(X)$ as the category of semi-stable coherent sheaves on $X$ of slope one. Denote $\mathcal{S}$ as the unique rank one object of $\mathsf{Sem}(X)$ which is not locally free. $\mathcal{A}$ is the rank two Atiyah sheaf; if $X$ is singular and $q\in X$ is a smooth point, then $\mathcal{B}_q$ is the (uniquely determined) indecomposable rank two locally free sheaf from $\mathsf{Sem}(X)$, whose determinant is $O([p]+[q])$ and whose Fourier-Mukai transform is supported at $s$. If $X$ is cuspidal then $\mathcal{U}$ is the unique indecomposable object of $\mathsf{Sem}(X)$ of rank two which is not locally free; there are two such objects $\mathcal{U}_{\pm}$ in the case $X$ is nodal.

We have the following description of torsion free sheaves of rank two from $\mathsf{Sem}(X)$:

\begin{theorem}{(\cite[Cor 2.10]{BZ})}
	Let $X$ be a singular Weierstrass  cubic curve, $p\in X$, $\mathbb{P}^1\xrightarrow{\nu}X$ the normalization morphism, $\mathcal{S}=\nu_{*}(\mathcal{O}_{\mathbb{P}^1})$ and $\mathcal{A}=\mathcal{A}_2$ the rank two Atiyah bundle on $X$. Let $\cf$ be a semi-stable torsion free sheaf on $X$ of rank two and slope one, $\mathcal{T}$ be the Fourier-Mukai transform (see in \cite{BZ}) of $\cf$ and $Z=\text{Supp}(\mathcal{T})$.
	\begin{enumerate}
		\item If $\cf$ is locally free and indecomposable, then it is either isomorphic to $\mathcal{A}\otimes \mathcal{O}(q)$ for some smooth point $q\in X$ or to $\mathcal{B}_{q}$, where $\det(\mathcal{B}_{q})=\mathcal{O}(q+p)\in \Pic^2 (X)$.
		\item If $\cf$ is  indecomposable but not locally free, then it is isomorphic to one of the sheaves $\mathcal{U}_{\pm}$(nodal case) or $\mathcal{U}$. In this case, $Z=\{s\}$.
		\item If is decomposable, then it is isomorphic to $\mathcal{O}(q)\oplus \mathcal{O}(q')$, $\mathcal{O}+\mathcal{S}$ or $\mathcal{S}\oplus \mathcal{S}$ for some smooth points $q,q'\in X$. We have: $Z=\{q,q'\}\text{or} \{q,s\}\text{or} \{s\}$ respectively.
	\end{enumerate}
	For any object $\cf$ of Sem$(X)$ we have: $H^1(X,\cf)=0$. Moreover, $\Gamma(X,\cf)\xrightarrow{ev|_p}\cf_p$ is an isomorphism iff $p\notin Z$.
\end{theorem}

The spectral sheaves from theorem \ref{T:parametrisation} are connected with sheaves from $\mathsf{Sem}(X)$ in a very simple way: $\cf\in \mathsf{Sem}(X)$ iff $\cf\otimes \co_X(-p)$ is spectral in the sense of theorem \ref{T:parametrisation}, i.e. $H^0(X, \cf\otimes \co_X(-p))=H^1(X, \cf\otimes \co_X(-p))=0$. 

We have the following description of spectral sheaves in terms of coefficients of the operator $L_4$:

\begin{theorem}{(\cite[Summary]{BZ})}
	\label{T:Rank 2 sheaves}
	
\begin{enumerate}
		\item The spectral curve $X$ is singular and $\mathcal{F}\cong \mathcal{S}\oplus \mathcal{S}$ iff $C_1=0$ and $C_0$ is a constant.
		\item Let $L_4$ be formally self-adjoint (i.e. $C_1=0$) with $C_0'\neq 0$. Then $C_0,C_2$ are given by 
		$$
		C_0=f\quad\text{and} \quad C_2=\frac{K_2+2K_3+f^3-f'''f'+\frac{1}{2}(f'')^2}{f'^2}
		$$
		for some $f\in \mathbf{C}[[x]]$ and $K_2,K_3\in \mathbf{C}$. We have in this case: $g_2=-2K_3,g_3=\frac{1}{2}K_2$. The spectral sheaf $\mathcal{F}$ is locally free and self-dual. Moreover, $Z={q_+,q_-}=\{(\lambda,\mu_+),(\lambda,\mu_-)\}$, where $\lambda=-\frac{1}{2}f(0),\mu_{\pm}^2=h(\lambda)$, with 
		\begin{enumerate}
			\item[(1)] If $q_+\neq q_-$, then $\mathcal{F}\cong \mathcal{O}(q_+)\oplus\mathcal{O}(q_-)$. 
			\item[(2)] If $q_+= q_-=q$ is a smooth point of $X$, then $\mathcal{F}\cong \mathcal{O}(q)\oplus \mathcal{O}(q)$ in the case $f'$ has zero of order three at $x=0$ and $\mathcal{F}\cong \mathcal{A}\otimes \mathcal{O}(q)$ otherwise.
			\item[(3)] If $X$ is singular and $Z=\{s\}$ then $\mathcal{F}\cong \mathcal{B}_q$.
		\end{enumerate}
		\item Assume now $C_1\neq 0$, i.e. $L_4$ is not self-adjoint case. Then $C_0,C_1,C_2$ are given by 
		$$\begin{cases}
			C_0=-f^2+K_{11}f+K_{12}
			\\C_1=f'
			\\C_2=\frac{K_{14}-2K_{10}f+6K_{12}f^2+2K_{11}f^3-f^4+f''^2-2f'f'''}{2f'^2}
		\end{cases}$$
		where $f\in x\mathbf{C}[[x]]$ and $K_{10},K_{11},K_{12},K_{14}\in \mathbf{C}$. The Weierstrass parameters $g_2,g_3$ of the spectral curve $X$ are given by the formulae:
		$$
		\begin{cases}
			g_2=3K_{12}^2+K_{10}K_{11}-K_{14}
			\\g_3=\frac{1}{4}(2K_{10}K_{11}K_{12}+4K_{12}^3+K_{14}(K_{11}^2+4K_12)-K_{10}^2)
		\end{cases}
		$$
		Consider the following expressions:
		$$
		\begin{cases}
			a(\lambda)=(\lambda+\frac{1}{2}K_{12})^2+\frac{1}{4}K_{14}
			\\b(\lambda)=(\lambda+\frac{1}{2}K_{12})K_{11}-\frac{1}{2}K_{10}
		\end{cases}
		$$
		Let $\lambda_1,\lambda_2$ be the roots of $a(\lambda)$. Then $Z=\{q_1,q_2\}=\{(\lambda_1,-b(\lambda_1)),(\lambda_2,-b(\lambda_2))\}$, with
		\begin{enumerate}
			\item[(1)] If $q_1\neq q_2$ (i.e. $K_{14}\neq 0$), then $\mathcal{F}\cong \mathcal{O}(q_1)\oplus\mathcal{O}(q_2)$. 
			\item[(2)] If $q_1= q_2=q$ is a smooth point of $X$(i.e. $K_{14}=0$ but $K_{10}\neq 0$), then $\mathcal{F}\cong \mathcal{O}(q)\oplus \mathcal{O}(q)$ in the case $f'$ has zero of order three at $x=0$ and $\mathcal{F}\cong \mathcal{A}\otimes \mathcal{O}(q)$ otherwise.
			\item[(3)] The spectral curve is singular and $Z={s}$ if and only if $K_{10}=K_{14}=0$, in this case 
			$$
			X=\overline{V\Big(y^2-4(x+\frac{K_{12}}{2})^2(x-K_{12})\Big)}
			$$
		    \begin{enumerate}
		    	\item[(a)] The spectral sheaf $\mathcal{F}$ is locally free iff $\Delta:=6K_{12}+K_{11}\neq 0$. Moreover, $\mathcal{F}\cong \mathcal{B}_q$ with $q=(\frac{1}{4}K_{11}^2+K_{12},\frac{1}{4}K_{11}(6K_{12}+K_{11}^2))$.
		    	\item[(b)] If $\Delta=0$ then $\mathcal{F}$ is indecomposable but not locally free. If $X$ is cuspidal(i.e. $K_{11}=K_{12}=0$) then $\mathcal{F}\cong \mathcal{U}$. If $X$ is nodal (i.e. $K_{12}\neq 0$) then $\mathcal{F}$ is isomorphic to one of the sheaves $\mathcal{U}_{\pm}$.
		    \end{enumerate}
		    \item[(4)] The spectral curve $X$ is singular and the spectral sheaf $\mathcal{F}$ is decomposable and not locally free iff $K_{10}=(3K_{12}+\frac{1}{2}K_{11}^2)K_{11}$ and $K_{14}=-(3K_{12}+\frac{1}{2}K_{11}^2)^2\neq 0$. In this case, $Z=\{s,q\}$ and $\mathcal{F}\cong \mathcal{S}\oplus \mathcal{O}(q)$, where $q=(-2K_{12}-\frac{1}{4}K_{11}^2,-\frac{1}{2}K_{11}(K_{11}^2+6K_{12}))$.
		\end{enumerate}
		\end{enumerate}
\end{theorem}

In this section we'll mostly concentrate on the equation $F(X,Y)=Y^2-X^3+a=0$. To match notations from \cite{BZ}, where the Weierstrass equation is given as $F(X,Y)=Y^2-4X^3+g_2X+g_3=0$, let's observe that if $(L_4,L_6)$ is a solution pair to the equation:
$$
Y^2-4X^3+4a=0
$$ 
then $(L_4,\tilde{L_6})$ (where $\tilde{L}_6=\frac{L_6}{2}$) satisfies $F(X,Y)=Y^2-X^3+a=0$. For convenience, we'll still denote $\tilde{L}_6$ as $L_6$, with the assumptions $g_2=0,g_3=4a$.

In the following sections we give all possible partially normalised normal forms for the curve  $f(X,Y)=Y^2-X^3+a=0$, the way is as follows:
\begin{enumerate}
	\item Calculate $L_6$ via the expression of $L_4$.
	\item Calculate the  Schur operator $S$ via the equation $L_4S=SD^4$.
	\item Calculate  the normal form $Y$ of $L_6$ with respect to $L_4$ via the equation $L_6S=SY$.
	\item Calculate the partial normalisation of the normal forms we get above.
	\item Determine the isomorphism classes of spectral sheaves via the equivalence classes of points (conjugacy classes of partially normalised normal forms).
\end{enumerate}

\subsection{The normal forms of $L_6$ with respect to $L_4$}
\label{S:normal_form_ofL_6}

Within this section, fix $k=4$ and denote $I:=\sqrt{-1}$.  

\begin{lemma}
	\label{L:L6}
Suppose $(L_4,L_6)$ is a pair of monic differential operators with the BC-polynomial  $$f(X,Y)=Y^2-X^3+a=0$$
where
\begin{equation}\label{L46}
\begin{cases}
L_4=(\partial^2+\frac{C_2(x)}{2})^2+2C_1(x)\partial +C'_1(x)+C_0(x);
\\L_6=\partial^6+T_4(x)\partial^4+T_3(x)\partial^3+T_2(x)\partial^2+
T_1(x)\partial+T_0(x);
\end{cases}
\end{equation}

then we have
\begin{equation}\label{Tix}
\begin{cases}
T_4(x) = \frac{3}{2} C_2(x);
\\T_3(x) = 3 C_1(x) + 3 C'_2(x);
\\T_2(x)=\frac{3}{2}C_{0}(x)+\frac{3}{4}C^{2}_2(x)+\frac{9}{2}C'_{1}(x)+\frac{7}{2}C^{(2)}_{2}(x);
\\T_1(x)=\frac{3}{2}C_1(x)C_2(x)+\frac{3}{2}C'_{0}(x)+\frac{3}{2}C'_2(x)C_2(x)+4C^{(2)}_1(x)+2C^{(3)}_2(x);
\\T_0(x)=\frac{3}{2}C^2_1(x)+\frac{3}{4}C_0(x)C_2(x)+\frac{1}{8}C^3_2(x)+\frac{3}{4}C'_1(x)C_2(x)+\frac{3}{4}C_1(x)C'_2(x)
\\\qquad+\frac{1}{2}C'^2_2(x)+\frac{5}{4}C^{(2)}_0(x)+\frac{3}{4}C^{(2)}_2(x)C_2(x)+\frac{5}{4}C^{(3)}_1(x)
+\frac{1}{2}C^{(4)}_2(x);
\end{cases}
\end{equation}
\end{lemma}
\begin{proof}
	This lemma is completed with the help of Wolfram Mathematica, here we only explain the first two steps of calculations: consider the coefficients at $\partial^{k}$ in $f(L_4,L_6)$ as the equations for $T_i(x)$.
	\begin{enumerate}
		\item Consider the coefficient at $\partial^{10}$, here only appears $T_{4}(x)$, namely
		$$2T_4(x)-3C_{2}(x)=0 $$ 
		Hence we get $T_4(x)$;
		\item Consider the coefficient at $\partial^{9}$, here $T_{4}(x),T_3(x)$ will appear, namely
		$$
		3 C_1(x) - T_3(x) + \frac{15}{2} C_2'(x) - 
		\frac{3}{2} T'_{4}(x)=0
		$$
		but we already know the expression of $T_4(x)$, so we get the expression of $T_{3}(x)$;
		\end{enumerate} 
Continuing this line of reasoning  we  get the expression of all $T_i(x)$.
\end{proof}

More precisely, in standard form (of HCPC), if we write
\begin{multline*}
	L_4=D^4+c_0 D^2+(c_1xD+c_2)D+(d_1x^2D^2+d_2xD+d_3)
	\\+(d_4x^3D^2+d_5x^2D+d_6x)+(d_7x^4 D^2+d_8x^3D+d_9x^2)+\dots
\end{multline*}
then by Lemma \ref{L:L6} $L_6$ should be in the (standard) form of 
\begin{multline*}
L_6=D^6+\frac{3}{2}c_0 D^4+(m_1xD+m_2)D^3+(m_3x^2D^2+m_4xD+m_5)D^2
\\+(m_6x^3D^3+m_7x^2D^2+m_8xD+m_9)D+
\\m_{10}x^4D^4+m_{11}x^3D^3+m_{12}x^2D^2+m_{13}xD+m_{14}+\dots
\end{multline*}
with the coefficient relation:
$$
\begin{cases}
	m_1=\frac{3}{2}c_1;m_2=\frac{3}{2}(c_1+c_2);
	\\m_{3}=\frac{3}{2}d_1;m_{4}=\frac{3}{2}(2d_1+d_2);m_5=\frac{3}{8}c_0^2+\frac{5}{2}d_1+\frac{3}{2}d_2+\frac{3}{2}d_3
	\\m_6 = \frac{3}{2} d_4;m_7 = \frac{9}{2}d_4 + \frac{3}{2} d_5;
	\\m_8 = \frac{3}{4} (c_0 c_1 + 2 (5 d_4 + 2 d_5 + d_6));m_9 = \frac{1}{4} (3 c_0 c_2 + 10 d_5 + 6 d_6);
	\\m_{10} = \frac{3}{2}d_7;m_{11} = 6 d_7 + \frac{3}{2}d_8;
	\\m_{12} = \frac{3}{8}c_1^2 + \frac{3}{4}c_0d_1+ 15 d_7 + \frac{9}{2}d_8 +\frac{3}{2}d_9;
	\\m_{13} = \frac{3}{2}b c_1 + \frac{3}{4}c_1^2 + \frac{3}{4}c_0d_2 + \frac{15}{2}d_8 + 
	3 d_9; 
	\\m_{14} = \frac{3}{2}b^2 - \frac{1}{16}c_0^3 + \frac{3}{4}b c_1 - \frac{1}{8} c_1^2 - \frac{1}{2}c_0d_1 + \frac{3}{4}c_0d_3 - 3 d_7 + \frac{5}{2}d_9,
\end{cases}
$$
where we set $b:=c_2-c_1$ to make formulae more readable. 
After that we transfer them into G-form.

Let $(L_{4},L_{6})$ be a solution pair to the equation  $F(X,Y)=Y^2-X^3+a=0$, with $\ord (L_{4})=4, \ord(L_{6})=6$. May assume that $L_{4}, L_6$ are monic ($L_{4}$ commutes with $L_{6}$ according to \cite{BC1}). According to \cite[Prop. 2.3]{GZ24}, there exists Schur operator $S$, with $\Ord (S)=0$, such that
$$
S^{-1}L_{4}S=\partial^{4}
$$
Let $H=S^{-1}L_{6}S$  be the normal form of $L_6$ with respect to $L_4$. Obviously  $H, \partial^4$ commute. Substituting this pair back to the original equation, we  get 
$$
S^{-1}L_{6}^{2}S=H^2=\partial^{12}-a
$$
On the other hand, according to \cite[Prop. 2.1]{GZ24}, $H$ should be of the form like
$$
H=D^6+\sum_{l=-3}^{4}\sum_{i=0}^3 c_{l,i}A_iD^l
$$
(here $D$ denotes either $\partial$ or $\int$ depending on the exponent).

\begin{lemma}
\label{L:2}
	Suppose an operator $H\in \hat{D}_{1}^{sym}\otimes_{K}\tilde{K}$ has the form above, then $H^2$ has the following shape:
$$
H^2= D^{12}+(2 c_{4,0}+ 2 A_2 c_{4,2}) D^{10}+ (2 c_{3,0} + 
2 A_2 c_{3,2}) D^9    + \dots
$$
\end{lemma}

This calculation is completed with the help of Wolfram Mathematica. 

Although the expression of $H^2$ is very long, we can deal with it step by step. First look at the coefficients at $D^9,D^{10}$ in $H^2=D^{12}-a$, we will get the following equations:
\begin{equation}\label{E ii}
\begin{cases}
2c_{4,0}+2c_{4,2}A_{2}=0 & (D^{10}) \\
2c_{3,0}+2c_{3,2}A_{2}=0 & (D^{9})
\end{cases}
\end{equation}
Hence, we have:
$$
c_{4,0}=0;c_{4,2}=0;c_{3,0}=0;c_{3,2}=0;
$$

This helps us  to calculate some part of the coefficients in $H$. Still the coefficients like $c_{41},c_{43},c_{31},c_{33}$ are unknown for us. Consider the Schur operator for $L_4$:
\begin{equation}\label{S}
S=S_{0}+S_{-1}+S_{-2}+\ldots \quad \Ord(S_{-i})=i,\forall i\in \mathbb{N}
\end{equation}
 By \cite[Prop. 2.3]{GZ24}  we can assume that $S_{0}=1$, $S_{-1}=0$. Then calculate $S_{-2},H_4$ and $S_{-3},H_1$, etc. step by step. For example from the homogeneous component of order $-2$ we will find $S_{-2}$ and $c_{41},c_{43}$. See the following Lemma:

\begin{lemma}{(Construction of $S_{-2}$ and $c_{4,1},c_{4,3}$)}
	\label{L:c41}

 Suppose $S$ is a Schur operator for $L_4$ such that $S^{-1}L_4S=D^4$, $S_{0}=1, S_{-1}=0$, $L_4,L_6$ are in the form of 
$$
\begin{cases}
L_4=D^4+c_0 D^2+\dots \\
L_6=D^6+\frac{3}{2}c_0 D^4+\dots
\end{cases}
$$ 
where \dots at $L_4$ means the terms whose $\Ord <2$, then we have
\begin{enumerate}
	\item $S_{-2}=\frac{1}{8}c_0((I-1)A_1-A_2+(-I-1)A_3-2\Gamma_1+3)D^{-2}$;
	\item $c_{4,1} = \frac{1 - I}{4} c_0;c_{4,3} =\frac{1 + I}{4} c_0.$
\end{enumerate}
\end{lemma}

\begin{proof}
Consider the equation $L_4S-SD^4=0$, in the following form:
	$$
	(D^4+c_0 D^2+\dots)(1+S_{-2}+\dots)-(1+S_{-2}+\dots)D^4=0
	$$
	Consider the equations of $\Ord =(-2)$, we get
	$$
	[\partial^4, S_{-2}]=-c_{0}D^2
	$$
	Following the calculation rules of  \cite[Lem. 2.5]{GZ24}, we get 
	$$
	S_{-2}=\frac{1}{8}c_0((I-1)A_1-A_2+(-I-1)A_3-2xD+3)D^{-2}
	$$

    Now  look back at the equation $L_4S-SH=0$, i.e. 
	$$
	(D^6+\frac{3}{2}c_0D^4+\dots)(1+S_{-2}+\dots)-(1+S_{-2}+\dots)(D^6+(c_{41}A_1+c_{43}A_3)D^4+\dots)=0
	$$
	Consider terms of $\Ord =4$, we get the equation
	$$
	D^6\cdot S_{-2}+\frac{3}{2}c_0D^4-S_{-2}\cdot D^6-(c_{41}A_1+c_{43}A_3)D^4=0
	$$
    Solve this equation (by viewing all coefficients of $1=A_0,A_1,A_2,A_3$ to be $0$), we have
	$$
	\begin{cases}
	c_0-2(c_{41}+c_{43})=0\\
	-Ic_{41}+c_{43}=0
	\end{cases}
	$$
	Hence we get 
	$$
	c_{41} = \frac{1 - I}{4} c_0;c_{43} =\frac{1 + I}{4} c_0. 
	$$
\end{proof}

\begin{rem}
	 Following this line of calculations, other homogeneous components $S_{i}$ with $i$ from $-2$ to $-7$ can be found,  and all calculations can be provided in the same way:
	\begin{enumerate}
		\item Find $m_{j}$ from the equation $L_6^2=L_4^3-a$.
		\item Calculate $S_{-i}$ from the corresponding homogeneous terms in equation $L_4S=SD^4$.
		\item Calculate $c_{i,j} $ from the homogeneous part of $\Ord=6+i$ in the equation $L_6S=SY$.
	\end{enumerate}
    To save space, we only give here the calculations for $S_{-2}$, and omit the rest of them for other $S_{i}$.

\end{rem}

Below we list all coefficients of $H$ (the normal forms of $L_6$ with respect to $L_4$), where we set again $b:=c_2-c_1$: 

$
\left\{\begin{aligned}
c_{4,0}&=0\\
c_{4,1}&=\frac{1-I}{4} c_0\\
c_{4,2}&=0\\
c_{4,3}&=\frac{1+I}{4} c_0 
\end{aligned}
\right.
$

$
\left\{\begin{aligned}
c_{3,0}&=0\\
c_{3,1}&=\frac{1-I}{2} b\\
c_{3,2}&=0 \\
c_{3,3}&=\frac{1+I}{2} b \\ 
\end{aligned}
\right.
$

$
\left\{\begin{aligned}
c_{2,0}&=-\frac{1}{8} c_0^2,
\\
c_{2,1}&=\frac{1}{4}\left((1+I) d_1-d_2+(1-I) d_3\right) \\
c_{2,2}&=0 \\
c_{2,3}&=\frac{1}{4}\left((1-I) d_1-d_2+(1+I) d_3\right) \\ 
\end{aligned}
\right.
$

$
\left\{\begin{aligned}
c_{1,0}&=-\frac{1}{4} b c_0,\\
c_{1,1}&=\frac{1}{8}\left(2 b c_0+c_0 c_1-6 I d_4+(2+2 I) d_5-2 d_6\right) \\
c_{1,2}&=\frac{1}{4} b c_0, \\
c_{1,3}&=\frac{1}{8}\left(2 b c_0+c_0 c_1+6 I d_4+(2-2 I) d_5-2 d_6\right) \\
\end{aligned}
\right.
$

$
\left\{\begin{aligned}
c_{0,0}&=0\\
c_{0,1}&=\frac{1-I}{32}[(8+8 I) b^2-c_0^3+4 b_1 +c_0(-8 d_1+(2-2 I) d_2+4 d_3)\\
&\qquad -2 I\left(c_1^2-24 I d_7+(6+6 I) d_8-4 d_9\right)] \\
c_{0,2}&=\frac{1}{2} b^2 \\
c_{0,3}&=\frac{1+I}{32}[(8-8 I) b^2-c_0^3+4 b_1 
+c_0(-8 d_1+(2+2 I) d_2+4 d_3)\\
&\qquad-2 I(c_1^2+24 I d_7+(6-6 I) d_8-4 d_9)]
\end{aligned}
\right.
$

$
\left\{\begin{aligned}
c_{-1,0}=&-\frac{1}{4}[c_{0}(c_{1,1}+c_{1,3})+2Ib(-c_{2,1}+c_{2,3})]
\\c_{-1, 1} =& 
\frac{1}{16}(2 b c_0^2 + c_0^2 c_1 + 2 b d_2 - 4 b d_3 + 2 c_0 d_5 - 2 c_0 d_6 + 
I (c_0^2 c_1 
\\&\qquad + 8 b d_1 + 12 c_1 d_1 + 120 d_10 - 12 d_12 - 10 b d_2 - 
3 c_1 d_2 + 8 b d_3 + 12 c_0 d_4 - 2 c_0 d_6))
\\c_{-1,2}=&-\frac{1}{4}[c_{0}(c_{1,1}+c_{1,3})-2b(c_{2,1}+c_{2,3})]
\\c_{-1,3} =& -c_{-1,0} - c_{-1,1} - c_{-1,2}
\end{aligned}\right.
$

$
\left\{\begin{aligned}
	c_{-2, 0}=&\frac{1}{128}(-64 b^2 c_0 + 3 c_0^4 - 32 b c_0 c_1 + 32 c_0^2 d_1 + 16 d_1^2 - 
	8 c_0^2 d_2 - 16 d_1 d_2 
	\\&\qquad + 8 d_2^2 - 16 c_0^2 d_3 - 16 d_2 d_3 + 16 d_3^2 + 
	96 b d_4 - 64 b d_5 + 32 b d_6 + 192 c_0 d_7 - 48 c_0 d_8)
	\\c_{-2,1}=&\frac{-1-I}{256}((-32 + 64 I) b^2 c_0 - 3 I c_0^4 + 32 I b c_0 c_1 + 8 c_0 c_1^2 - 
	32 I c_0^2 d_1- 16 I d_1^2+ (8 + 8 I) c_0^2 d_2
	\\&\qquad  - (16 - 16 I) d_1 d_2 + (8 - 
	8 I) d_2^2  + 16 I c_0^2 d_3 + 32 d_1 d_3 - (16 - 16 I) d_2 d_3
	\\&\qquad - 16 I d_3^2 - 96 I b d_4 + 64 I b d_5 - 32 I b d_6 - 
	192 I c_0 d_7 + (48 + 48 I) c_0 d_8 - 32 c_0 d_9)
	\\c_{-2,2} =& I c_{-2,0} - (1 - I) c_{-2,1}
	\\c_{-2,3} =& (-1 - I) c_{-2,0} - I c_{-2,1}
\end{aligned}\right.
$

$
\left\{\begin{aligned}
c_{-3,0} =& \frac{1}{32}(-4 b^2 c_1 - 2 b c_1^2 + 4 b c_0 d_1 - 2 c_0 c_1 d_1 - 2 b c_0 d_2 + 
c_0 c_1 d_2 - 4 b c_0 d_3
\\&\qquad - 6 d_2 d_4 + 12 d_3 d_4 - 4 d_1 d_5 + 4 d_2 d_5 - 
4 d_3 d_5 + 4 d_1 d_6 - 2 d_2 d_6 + 48 b d_7 
\\&\qquad - 24 b d_8 + 8 b d_9+ I (2 b c_0^3 + c_0^3 c_1 + 2 b c_0 d_2 - 4 b c_0 d_3 + 2 c_0^2 d_5 - 
2 c_0^2 d_6 - 16 c_0 c_{-1, 1}))
\\c_{-3,1} =& I c_{-3,0}\\
c_{-3,2} =& -c_{-3,0} \\
c_{-3,3} =& -I c_{-3,0}\\
\end{aligned}
\right.
$

\subsection{Self-adjoint case}
\label{S:self-adjoint}

To simplify notations, and to be able to apply our calculations also to the examples of papers by authors of \cite{MZh}, \cite{SMZh}, we may follow the expression of $L_4$ as in loc. cit.: in self-adjoint case, $L_4=(D^2+\frac{V(x)}{2})^2+W(x)$ with
$$
\begin{cases}
V(x)=v_0+v_1x+\dots
\\W(x)=w_0+w_1x+\dots
\end{cases}
$$

Then the  formulae for $C_2$ from theorem \ref{T:Rank 2 sheaves} become:
\begin{equation}
	\label{E:Grum self case}
V=\frac{8a+W^3-W'''W'+\frac{1}{2}(W'')^2}{W'^2}
\end{equation}
Comparing old coefficients of $L_4$ from the previous section with the coefficients of $v_i,w_j$,  we get the following dictionary:
$$
\begin{cases}
b = 0;c_0 = 2 v_0; c_1 = 2 v_1;c_2 = 2 v_1;
\\ d_1 = 2 v_2; d_2 = 4 v_2; d_3 = v_0^2 + 2 v_2 + w_0;
\\ d_4 = 2 v_3;d_5 = 6 v_3; d_6 = (2 v_0 v_1 + 6 v_3 + w_1) ;
\\ d_7 = 2 v_4; d_8 = 8 v_4;d_9 = (v_1^2 + 2 v_0 v_2 + 12 v_4 + w_2);
\\ d_{10}= v_5;d_{11}=5v_5;d_{12}=\frac{1}{2} (v_1v_2+v_0v_3+20v_5+2 w_3)
\end{cases}
$$
Substituting all these coefficient relations back into $H^2$, we get
$$
H^2=D^{12}+\frac{1}{8}(w_0^3 - v_0 w_1^2 + 2 w_2^2 - 6 w_1 w_3)
$$ 

Hence we get the additional relation 
\begin{equation}
	\label{E:w in Self adjoint}
w_0^3 -  v_0 w_1^2 + 2 w_2^2 - 6 w_1 w_3=-8a
\end{equation}

Now we separate the calculations according to the zero order $\nu$ of $W'(x)$, following results from \cite{BZ}.

\subsubsection{$\nu=0$ case}
\label{S:nu0case}

Assume $\nu=0$, i.e. $w_1\neq 0$, we can solve the equation \eqref{E:w in Self adjoint} (notice that this solution coincides with formula \eqref{E:Grum self case})
$$v_0=\frac{8 a+w_0^3+2 w_2^2-6 w_1 w_3}{w_1^2}$$

Let's find the partial normalisation of $H$. Recall that by lemma \ref{L:normalisation}  we can choose invertible $T\in C(\partial^4)$ so that $\tilde{H}=THT^{-1}$ is partially normalised. For example we can choose in our case  $T=(T_2T_1)^{-1}$, with
\begin{equation}
	\label{E:T12}
T_1=1+\sum_{l=-3}^{-1}(\sum_{i=0}^3s_{l,i}A_i)D^l \quad \text{and} \quad T_2=1+s_{0,2}A_2
\end{equation}
with 
\begin{align}
	\label{E:T1 invertible}
\left\{\begin{aligned}
	s_{-1,3}&=-(s_{-1,0}+s_{-1,1}+s_{-1,2});
	\\s_{-2, 2}& = I (s_{-2, 0} + (1 + I) s_{-2, 1}); 
	\\s_{-2, 3} &= (-1 - I) s_{-2, 0} - I s_{-2, 1};
	\\s_{-3,1}&=I s_{-3,0};
	\\s_{-3,2}&=-s_{-3,0};
	\\s_{-3,3}&=-I s_{-3,0};
\end{aligned}\right\}
\end{align}

\begin{align}
	\label{E:Second conjugation}
	\left\{\begin{aligned}
		s_{-1,1}&=0;s_{-1,2}=-s_{-1,0};
		\\s_{-2,0}&=\frac{v_0}{4};s_{-2,1}=(-\frac{1}{8}+\frac{I}{8}) v_0;
		\\s_{-3,0}&=\frac{1}{4} v_0s_{-1,0}; s_{-1,0}=-\frac{w_2}{2 w_1};
		\\s_{0,2}&=\frac{4-w_1}{w_1+4}\quad \text{or}\quad 0;
		\end{aligned}\right\}
\end{align}

\begin{rem}
Condition (\ref{E:T1 invertible}) is equivalent to the invertibility of $T_1$. In condition (\ref{E:Second conjugation}), the first row of assignments makes the coefficient at $\partial^5$ in $\tilde{H}$ to be $0$, the second row  makes the coefficient at $\partial^4$ in $\tilde{H}$ to be $0$, the third row  makes the coefficient at $\partial^3$ in $\tilde{H}$ to be $0$. The fourth row is to modify the coefficients which multiply with $w_1$ in $\partial^2$ and $\partial^{-2}$, here in the case when $w_1=-4$ we don't need $T_2$, so assume $s_{0,2}=0$(i.e. $T_2=1$).
\end{rem}

For convenience, we'll denote the new $\tilde{H}$ as $H$. So, we have 
\begin{multline*}
H=D^6+\Big(\frac{1}{4}+\frac{I}{4}\Big)\Big(w_0 (A_3-I A_1)\Big)D^2+(-A_1-A_3)D
\\-\frac{I}{32} (A_1-A_3) (8 a+w_0^3)D^{-1}+\Big(\frac{1}{16}+\frac{I}{16}\Big)\Big( w_0^2 ((1-I)+I A_1-A_3)\Big)D^{-2}
\end{multline*}

We have 3 possibilities here:
\begin{enumerate}
	\item If $w_0^3+8a\neq 0$, in this situation $\mu_{\pm}^2=-\frac{w_0^3}{2}-4a\neq 0$, hence $q_{+}\neq q_{-}$. We have $\mathcal{F}\cong \mathcal{O}(q_+)\oplus\mathcal{O}(q_-)$, with $q_+=(-\frac{w_0}{2},\frac{\sqrt{-8a-w_0^3}}{2}),q_-=(-\frac{w_0}{2},\frac{-\sqrt{-8a-w_0^3}}{2})$. The matrix form is like:
	$$
	\left(
	\begin{array}{cccc}
		0 & -2 & D^4+\frac{w_0}{2} & 0 \\
		\frac{1}{16} (8 a+w_0^3) & 0 & 0 & D^4+\frac{w_0}{2} \\
		\frac{1}{4} ( 4D^8-2 w_0D^4+w_0^2)& 0 & 0 & 2 \\
		0 & \frac{1}{4} \left(4 D^8-2 w_0D^4 +w_0^2\right) & \frac{1}{16} \left(-8 a-w_0^3\right) & 0 \\
	\end{array}
	\right)
	$$
	\item If $w_0^3+8a=0$, but $a\neq 0$, then $q_+=q_-=q$ is the smooth point of $X$. We have $\mathcal{F}\cong \mathcal{O}(q)\otimes \mathcal{A}$, with $q=(-\frac{w_0}{2},0)$. The matrix form is like:
	$$
	\left(
	\begin{array}{cccc}
		0 & -2 & D^4-\sqrt[3]{a} & 0 \\
		0 & 0 & 0 & D^4-\sqrt[3]{a} \\
		a^{2/3}+\sqrt[3]{a} D^4+D^8 & 0 & 0 & 2 \\
		0 & a^{2/3}+\sqrt[3]{a} D^4+D^8 & 0 & 0 \\
	\end{array}
	\right)
	$$
	\item If $w_0=a=0$, then $X$ is singular, $q=(0,0)$. We have $\mathcal{F}\cong \mathcal{B}_{q}$. The matrix form is like:
	$$
	\left(
	\begin{array}{cccc}
		0 & -2 & D^4 & 0 \\
		0 & 0 & 0 & D^4 \\
		D^8 & 0 & 0 & 2 \\
		0 & D^8 & 0 & 0 \\
	\end{array}
	\right)
	$$ 
\end{enumerate}

\subsubsection{$\nu=1$ case}
\label{S:nu1}

Assume $\nu=1$, i.e. $w_1=0$ but $w_2\neq 0$, then  condition \eqref{E:w in Self adjoint} becomes:
$w_0^3+2 w_2^2=-8a$. And  formulae \eqref{E:Grum self case} becomes 
$$
v_0=\frac{3 (w_0^2-8 w_4)}{4 w_2}
$$

Here we also use $T_1,T_2$ determined by \eqref{E:T12} to partially normalise  $H$, with  condition \eqref{E:T1 invertible} and 

\begin{align}
	\label{E:Second conjugation self v=1}
	\left\{\begin{aligned}
		s_{-1,1}&=0;s_{-1,2}=-s_{-1,0};
		\\s_{-2,0}&=\frac{v_0}{4};s_{-2,1}=(-\frac{1}{8}+\frac{I}{8}) v_0;
		\\s_{-3,0}&=\frac{1}{4} v_0s_{-1,0}; s_{-1,0}=-\frac{3 w_3}{4 w_2};
		\\s_{0,2}&= 0;
		\end{aligned}\right\}
\end{align}

Again, still written $\tilde{H}$ as $H$, we get
\begin{multline*}
H=D^6+(\frac{1}{4}+\frac{I}{4})w_0 (A_3-I A_1)D^2
\\+(\frac{1}{4}+\frac{I}{4}) w_2 (A_1-I A_3)+(\frac{1}{16}-\frac{I}{16})w_0^2 ((1+I)-A_1-I A_3)D^{-2}
\end{multline*}

Since $w_2\neq 0$, we know $w^3_0+8a=-2w_2^2\neq 0$, thus $q_+\neq q_-$, thus $\mathcal{F}\cong \mathcal{O}(q_+)\oplus \mathcal{O}(q_-)$, where $q_+=(-\frac{w_0}{2},\frac{\sqrt{-8a-w_0^3}}{2}),q_-=(-\frac{w_0}{2},\frac{-\sqrt{-8a-w_0^3}}{2})$. The matrix form of $H$ looks like:
$$
\left(
\begin{array}{cccc}
	\frac{\sqrt{-8 a-w_0^3}}{2 \sqrt{2}} & 0 & D^4+\frac{w_0}{2} & 0 \\
	0 & -\frac{\sqrt{-8 a-w_0^3}}{2 \sqrt{2}} & 0 & D^4+\frac{w_0}{2} \\
	D^8-\frac{1}{2} D^4 w_0+\frac{w_0^2}{4} & 0 & -\frac{\sqrt{-8 a-w_0^3}}{2 \sqrt{2}} & 0 \\
	0 & D^8-\frac{1}{2} D^4 w_0+\frac{w_0^2}{4} & 0 & \frac{\sqrt{-8 a-w_0^3}}{2 \sqrt{2}} \\
\end{array}
\right)
$$

\subsubsection{$\nu=2$ case}
\label{S:nu2}

Assume $\nu=2$, i.e. $w_1=w_2=0$, but $w_3\neq 0$, then condition \eqref{E:w in Self adjoint} becomes: $w_0^3=-8a$. For the same, use $T_1,T_2$ determined by \eqref{E:T12} to partially normalise $H$, with  condition \eqref{E:T1 invertible} and 
\begin{align}
	\label{E:Second conjugation self v=2}
	\left\{\begin{aligned}
		s_{-1,1}&=0;s_{-1,2}=-s_{-1,0};
		\\s_{-2,0}&=\frac{v_0}{4};s_{-2,1}=(-\frac{1}{8}+\frac{I}{8}) v_0;
		\\s_{-3,0}&=\frac{1}{4} v_0s_{-1,0}; s_{-1,0}=1;
		\\s_{0,2}&= \frac{ 3 w_3-4}{3 w_3+4}\quad \text{or} \quad 0;
		\end{aligned} \right\}
\end{align}

We get
\begin{multline*}
	H=D^6+(\frac{1}{4}+\frac{I}{4})w_0 (A_3-I A_1)D^2
	\\-I(A_1-A_3)D^{-1}+(\frac{1}{16}+\frac{I}{16}) w_0^2 ((1-I)+I A_1-A_3)D^{-2}
\end{multline*}

Since $w_0^3+8a=0, q_{+}=q_-$, we only have two possibilities here
\begin{enumerate}
	\item $a\neq 0$, then $q\neq s$ is the smooth point of $X$. In this case $\mathcal{F}=\mathcal{A}\otimes \mathcal{O}(q)$, $q=(-\frac{w_0}{2},0)$. The matrix form looks like:
	$$
	\left(
	\begin{array}{cccc}
		0 & 0 & D^4-\sqrt[3]{a} & 0 \\
		2 & 0 & 0 & D^4-\sqrt[3]{a} \\
		a^{2/3}+\sqrt[3]{a} D^4+D^8 & 0 & 0 & 0 \\
		0 & a^{2/3}+\sqrt[3]{a} D^4+D^8 & -2 & 0 \\
	\end{array}
	\right)
	$$
	\item $a=0$, the curve is singular and $\cf$ is locally free at $q=(0,0)$, $\mathcal{F}\cong \mathcal{B}_q$. The matrix form looks like:
	$$
	\left(
	\begin{array}{cccc}
		0 & 0 & D^4 & 0 \\
		2 & 0 & 0 & D^4 \\
		D^8 & 0 & 0 & 0 \\
		0 & D^8 & -2 & 0 \\
	\end{array}
	\right)
	$$
	
\end{enumerate}

\subsubsection{$\nu =  3$ case}
\label{S:nu3}

Assume $\nu =3$, i.e. $w_1=w_2=w_3=0$, then  condition \eqref{E:w in Self adjoint} becomes: $w_0^3=-8a$. For the same, use $T_1$ determined by \eqref{E:T12} to partially normalise $H$, with condition \eqref{E:T1 invertible} and 
\begin{align}
	\label{E:Second conjugation self v=3}
	\left\{\begin{aligned}
		s_{-1,1}&=0;s_{-1,2}=-s_{-1,0};
		\\s_{-2,0}&=\frac{v_0}{4};s_{-2,1}=(-\frac{1}{8}+\frac{I}{8}) v_0;
		\\s_{-3,0}&=\frac{1}{4} v_0s_{-1,0}; s_{-1,0}=1;
		\\s_{0,2}&= 0;
		\end{aligned}\right\}
\end{align}

We get 
$$
H=D^6+(\frac{1}{4}+\frac{I}{4})w_0 (A_3-I A_1)D^2+(\frac{1}{16}+\frac{I}{16}) w_0^2 ((1-I)+I A_1-A_3)D^{-2}
$$

We have two possibilities here:
\begin{enumerate}
	\item $a\neq 0$, here since $\nu=3$, $q_+=q_-=q=(-\frac{w_0}{2},0)$ and $\mathcal{F}=\mathcal{O}(q)\oplus\mathcal{O}(q)$. The matrix form looks like 
	$$
	\left(
	\begin{array}{cccc}
		0 & 0 & D^4-\sqrt[3]{a} & 0 \\
		0 & 0 & 0 & D^4-\sqrt[3]{a} \\
		a^{2/3}+\sqrt[3]{a} D^4+D^8 & 0 & 0 & 0 \\
		0 & a^{2/3}+\sqrt[3]{a} D^4+D^8 & 0 & 0 \\
	\end{array}
	\right)
	$$
	\item $a=0$ so that $W(x)$ becomes a constant and $H=\partial^6$. $q=s=(0,0)$ and $\mathcal{F}\cong \mathcal{S}\oplus \mathcal{S}$. The matrix form of $H$ is 
	$$\left(
	\begin{array}{cccc}
		0 & 0 & D^4 & 0 \\
		0 & 0 & 0 & D^4 \\
		D^8 & 0 & 0 & 0 \\
		0 & D^8 & 0 & 0 \\
	\end{array}
	\right)$$
\end{enumerate}

\subsection{Non-self adjoint case}
\label{S:non-self_adj}

In non-self adjoint case, we use the expression  $L_4=(D^2+\frac{V(x)}{2})^2+R(x)D+DR(x)+W(x)$, with 
$$
\begin{cases}
	V(x)=v_0+v_1x+\cdots
	\\W(x)=w_0+w_1x+\cdots
	\\R(x)=r_0+r_1x+\cdots
\end{cases}
$$

Then the formulae from theorem \ref{T:Rank 2 sheaves} become ($\exists f\in x\mathbf{C}[[x]]$):
\begin{align}
	\label{E:Gr non self}
	\left\{\begin{aligned}
	W&=-f^2+K_{11}f+K_{12}
	\\R&=f'
	\\V&=\frac{K_{14}-2K_{10}f+6K_{12}f^2+2K_{11}f^3-f^4+f''^2-2f'f'''}{2f'^2}
	\end{aligned}\right\}
\end{align}
with 
$$
\begin{cases}
	g_2=3K_{12}^2+K_{10}K_{11}-K_{14}=0
	\\g_3=\frac{1}{4}(2K_{10}K_{11}K_{12}+4K_{12}^3+K_{14}(K_{11}^2+4K_{12})-K_{10}^2)=4a
\end{cases}
$$
Also denote $\nu$ as the zero order of $f'$, and substitute the coefficients of $V,W,R$ to $L_4$, we get
$$
\begin{cases}
	c_0=v_0; c_1=v_1;b_1=f_1;
	\\ d_1 = v_2; d_2 = 2 (2 f_2 + v_2); d_3 = \frac{1}{4} (4 K_{12} + 8 f_2 + v_0^2 + 4 v_2);
	\\d_4 = v_3; d_5 = 3 (2f_3 + v_3); d_6 = \frac{1}{2} (2 K_{11} f_1 + 12 f_3 + v_0 v_1 + 6 v_3);
	\\d_7 = v_4; d_8 = 4 (2 f_4 + v_4); 
	\\d_9 = 
	\frac{1}{4} (-4 f_1^2 + 4 K_{11} f_2 + 48 f_4 + v_1^2 + 2 v_0 v_2 + 
	24 v_4);
	\\d_{10} = v_5; d_{11} = 5 (2 f_5 + v_5); 
	\\d_{12} = \frac{1}{2} (-4 f_1 f_2 + 2 K_{11} f_3 + 40 f_5 + v_1 v_2 + v_0 v_3 + 20 v_5);
\end{cases}  
$$

Substitute into $H^2$, we get
\begin{equation}
	\label{E:f in non self adjoint}
H^2=D^{12}+\frac{1}{2} f_1 (3 K_{12} f_1 - 60 f_5 - 6 f_3 v_0 - 3 f_2 v_1 - f_1 v_2)A_2D^2+\cdots
\end{equation}

Since $H^2=D^{12}-a$, thus we get many equations on coefficients of HPCs (to save space we don't list all of them here). 

Now we separate the discussion according to the zero order $\nu$ of $f'(x)$.

\subsubsection{$\nu=0$ case}

Assume $\nu=0$, i.e. $f_1\neq 0$, solve \eqref{E:f in non self adjoint}, we get
$$
\begin{cases}
	f_3=\frac{1}{12 f_1}(-4 f_2 K_{11} v_0-f_1 K_{11} v_1-24 f_4 K_{11}-2 f_1^2 v_0+4 f_2^2+3 K_{12}^2);
	\\f_4=\frac{1}{24} (-4 f_2 v_0-f_1v_1-K_{10});
	\\f_5=\frac{1}{60} (3 f_1 K_{12}-f_1 v_2-6 f_3 v_0-3 f_2 v_1);
\end{cases}
$$

Use $T_1,T_2$ determined by \eqref{E:T12} to partially normalise $H$, with condition \eqref{E:T1 invertible} and 
\begin{align}
	\label{E:Second conjugation nonself v=0}
	\left\{\begin{aligned}
		s_{-1,1}&=0;s_{-1,2}=-s_{-1,0};
		\\s_{-2,0}&=\frac{v_0}{4};s_{-2,1}=(-\frac{1}{8}+\frac{I}{8}) v_0;
		\\s_{-3,0}&=\frac{2 (2 f_1^2-f_2 v_0)+I K_{12} v_0}{16 f_1};
		\\ s_{-1,0}&=\frac{I (K_{12}+2 I f_2)}{4 f_1};
		\\s_{0,2}&=\frac{2-f_1 }{f_1+2}\quad \text{or} \quad 0;
		\end{aligned}\right\}
\end{align}

The modified  expression of $H$:
\begin{multline*}
H=D^6+(A_1+A_3)D^3+\frac{1-I}{2}K_{12}A_1D^2
\\+\Big(\frac{1}{16} I (2 K_{12}^2+(K_{10}+8 I) K_{11})A_1-\frac{1}{16} I (2 K_{12}^2+(K_{10}-8 I) K_{11})A_3 \Big)D
\\+(-\frac{1}{8}+\frac{I}{8}) (K_{10}+K_{11}K_{12})A_1+(-\frac{1}{8}-\frac{I}{8}) (K_{10}-K_{11}K_{12})A_3
\\+\Big(\frac{K_{12}}{2}-\frac{I}{32}(K_{10} K_{11}^2 + 2 K_{12} (-8 I + K_{11} K_{12})A_1+\frac{I}{32}(K_{10} K_{11}^2 + 2 K_{12} (8 I + K_{11} K_{12})A_3)\Big)D^{-1}
\\+\Big(\frac{1}{8} (-K_{10} K_{11} - 2 K_{12}^2)+\frac{1-I}{16}K_{10}K_{11}A_1+\frac{1+I}{16}(4K_{12}^2+K_{10}K_{11})A_3\Big)D^{-2}
\\-\frac{1}{32}K_{12} (K_{10} K_{11} + 2 K_{12}^2)(1+IA_1-A_2-IA_3)D^{-3} 
\end{multline*}

Here we have 3 possibilities:
\begin{enumerate}
	\item $K_{14}=3K_{12}^2+K_{10}K_{11}\neq 0$, then $q_1\neq q_2$, the sheaf $\mathcal{F}\cong\mathcal{O}(q_1)\oplus \mathcal{O}(q_2)$, where $q_{1,2}=\frac{1}{2} (\sqrt{-K_{14}}-K_{12}),\frac{1}{2} (K_{10}\pm K_{11} \sqrt{-K_{14}})$.
	 The matrix form is like:
	$$\begin{pmatrix}
	G_1	& G_2 \\
	G_3	& -G_1
	\end{pmatrix}$$
	where 
	$$
	G_1=\left(
	\begin{array}{cc}
		\frac{1}{4} (-K_{10}+I K_{11} K_{12}) & -K_{11} \\
		\frac{1}{16} K_{11} \left(K_{10} K_{11}+2 K_{12}^2\right) & \frac{1}{4} (-K_{10}-I K_{11} K_{12}) \\
	\end{array}
	\right)
	$$
	$$
	G_2=\left(
	\begin{array}{cc}
		D^4+\left(\frac{1}{2}-\frac{I}{2}\right) K_{12} & 2 \\
		\frac{1}{8} \left(-K_{10} K_{11}-2 K_{12}^2\right) & D^4+\left(\frac{1}{2}+\frac{I}{2}\right) K_{12} \\
	\end{array}
	\right)
	$$
	 $$
	G_3=\left(
	\begin{array}{cc}
	\Omega_1 & 2 K_{12}-2 D^4 \\
		\frac{1}{8} (D^4-K_{12}) \left(K_{10} K_{11}+2 K_{12}^2\right) & \Omega_2 \\
	\end{array}
	\right)
	$$
	with $\Omega_1=\frac{1}{4} \left(-K_{10} K_{11}-\left((2+2 I) K_{12}^2\right)-(2-2 I) K_{12} D^4+4 D^8\right)$,
	\\$\Omega_2=\frac{1}{4} \left(-K_{10} K_{11}-\left((2-2 I) K_{12}^2\right)-(2+2 I) K_{12} D^4+4 D^8\right)$.
	\item $K_{14}=0$, but $K_{10}\neq 0$, then $q_1=q_2=(-\frac{K_{12}}{2},\frac{K_{10}}{2})$, and $\mathcal{F}\cong \mathcal{A}\otimes \mathcal{O}(q)$ (after substitution $K_{11}\rightarrow -\frac{3K_{12}^2}{K_{10}}$). The matrix form looks like
	$$\begin{pmatrix}
		G_4	& G_5 \\
		G_6	& -G_4
	\end{pmatrix}$$
	where 
	$$
	G_4=\left(
	\begin{array}{cc}
		-\frac{K_{10}^2+3 I K_{12}^3}{4 K_{10}} & \frac{3 K_{12}^2}{K_{10}} \\
		\frac{3 K_{12}^4}{16 K_{10}} & -\frac{K_{10}}{4}+\frac{3 I K_{12}^3}{4 K_{10}} \\
	\end{array}
	\right)
	$$
	$$
	G_5=\left(
	\begin{array}{cc}
		D^4+\left(\frac{1}{2}-\frac{I}{2}\right) K_{12} & 2 \\
		\frac{K_{12}^2}{8} & D^4+\left(\frac{1}{2}+\frac{I}{2}\right) K_{12} \\
	\end{array}
	\right)
	$$
	$$
	G_6=\left(
	\begin{array}{cc}
		\frac{1}{4} \left((1-2 I) K_{12}^2-(2-2 I) K_{12} D^4+4 D^8\right) & 2 K_{12}-2 D^4 \\
		\frac{1}{8} K_{12}^2 (K_{12}-D^4) & \frac{1}{4} \left((1+2 I) K_{12}^2-(2+2 I) K_{12} D^4+4 D^8\right) \\
	\end{array}
	\right)
	$$
	\item $K_{14}=K_{10}=0$, thus $K_{12}=0$ so that $K_{11}=0$, $X$ is cuspidal and $q=(0,0)$, the sheaf $\mathcal{F}\cong \mathcal{U}$. The matrix form looks like
	$$
	\left(
	\begin{array}{cccc}
		0 & 0 & D^4 & 2 \\
		0 & 0 & 0 & D^4 \\
		D^8 & -2 D^4 & 0 & 0 \\
		0 & D^8 & 0 & 0 \\
	\end{array}
	\right)
	$$
\end{enumerate}

\subsubsection{$\nu=1$ case}

Assume now $\nu=1$, i.e. $f_1=0, f_2\neq 0$, solve the equation $H^2=D^{12}-a$ by each HCP part, at the coefficient of $D^3$, we get 
$$
f_5=\frac{1}{20} (-2 f_3 v_0-f_2 v_1);
$$

Look into the coefficient of $A_2$ in $D^0$, we get terms of $K_{11}f_4$, so we should separate the discussion into $K_{11}=0$ and $K_{11}\neq 0$ case. 

We first discuss $K_{11}\neq 0$ case, we have 
$$
f_4=\frac{-4f_2 K_{11} v_0+4 f_0^2+3K_{12}^2}{24 K_{11}};
$$

Now look at the coefficient of $1=A_{4,0}$ in $D^0$, we get a 4th degree equation of $f_2$, solve it we get 
$$
f_2=\pm \frac{1}{2} \sqrt{-K_{10} K_{11}-3K_{12}^2}\quad \text{or}\quad \pm\frac{1}{2} \sqrt{K_{10} K_{11}-K_{11}^4-6 K_{11}^2 K_{12}-3 K_{12}^2}
$$

We should notice that, only the first two cases fit with \eqref{E:Gr non self} (since $V(x)\in \mathbf{C}[[x]]$), so we only consider $f_2=\pm\frac{1}{2} \sqrt{-K_{10} K_{11}-3K_{12}^2}$, (notice that $K_{14}=K_{10} K_{11}+3K_{12}^2$, so that $K_{14}$ can't be equal to $0$.)

Use $T_1,T_2$ determined by \eqref{E:T12} to partially normalise $H$, with condition \eqref{E:T1 invertible} and 
\begin{align}
	\label{E:Second conjugation nonself v=1}
	\left\{\begin{aligned}
		s_{-1,1}&=0;s_{-1,2}=-s_{-1,0};
		\\s_{-2,0}&=\frac{v_0}{4};s_{-2,1}=(-\frac{1}{8}+\frac{I}{8}) v_0;
		\\s_{-3,0}&=\frac{1}{4} v_0 s_{-1,0};
		\\ s_{-1,0}&=-\frac{3 f_3}{4 f_2};
		\\s_{0,2}&= 0;
		\end{aligned}\right\}
\end{align}

The modified  expression of $H$:
\begin{multline*}
	H=D^6+\left(\frac{1}{4}-\frac{I}{4}\right) (-2 I f_2 A_1-2 f_2 A_3+K_{12} A_1+I K_{12} A_3)D^2
	\\+\frac{1+I}{8 K_{11}} \left(2 f_2 \left(2 f_2 (A_3-I A_1)+K_{11}^2 (A_1-I A_3)\right)+3 K_{12}^2 (A_3-I A_1)\right)D
	\\+\frac{1}{16} \Big((4+4 I) f_2 K_{12} (A_1-(1-I) A(4,2)-I A_3)
	\\+(4+4 I) f_2^2 ((1-I) A(4,0)+I A_1-A_3)+K_{12}^2 (2 A(4,0)-(1-I) (A_1+I A_3))\Big)D^{-2}
\end{multline*}

Since $K_{14}\neq 0$ in this case, we know the sheaf $\mathcal{F}\cong \mathcal{O}(q_1)\oplus \mathcal{O}(q_2)$, with $q_{1,2}=\frac{1}{2} (\sqrt{-K_{14}}-K_{12}),\frac{1}{2} (K_{10}\pm K_{11} \sqrt{-K_{14}}))$. The matrix form looks like
$$
\begin{pmatrix}
G_7	& G_8 \\
G_9	& -G_7
\end{pmatrix}
$$
where 
$$
G_7=\begin{pmatrix}
\frac{2 f_2 \left(2 f_2+K_{11}^2\right)+3 K_{12}^2}{4 K_{11}}	& 0 \\
0	& \frac{-2 f_2 K_{11}^2+4 f_2^2+3 K_{12}^2}{4 K_{11}}
\end{pmatrix}
$$
$$
G_8=\begin{pmatrix}
D^4-f_2+\frac{K_{12}}{2}	&  0\\
0	& D^4+f_2+\frac{K_{12}}{2}
\end{pmatrix}
$$
$$
G_9=\begin{pmatrix}
D^8+D^4 \left(f_2-\frac{K_{12}}{2}\right)+\frac{1}{4} (K_{12}-2 f_2)^2	& 0 \\
0	& D^8-\frac{1}{2} D^4 (2 f_2+K_{12})+\frac{1}{4} (2 f_2+K_{12})^2
\end{pmatrix}
$$

Secondly, consider the case when $K_{11}=0$, we have 
$$
f_4=\frac{1}{24} (K_{10}-4 f_2 v_0); f_2=\pm \frac{\sqrt{3} K_{12}}{2 I};
$$
In this case also $K_{14}\neq 0$, and we see that although the expression of $f_4$ are different from the first case, but the matrix form are the same (after substitution $K_{11}=0$ into the expression of $H$ in the first case). So we just uniform these two cases.

\subsubsection{$\nu=2$ case}

Assume $\nu=2$, then we have $f_1=f_2=0,f_3\neq 0$, consider \eqref{E:Gr non self}, here 
\begin{multline*}
	K_{14}-2K_{10}f+6K_{12}f^2+2K_{11}f^3-f^4+f''^2-2f'f'''
	\\=(K_{10} K_{11}+3 K_{12}^2)-2 x^3 (f_3 (24 f_4+K_{10}))+(\cdots)x^4
\end{multline*}
with
$$
2f'^2=18f_3^2x^4+\cdots
$$
Since $V(x)\in \mathbf{C}[[x]]$, we know $K_{14}=K_{10} K_{11}+3 K_{12}^2=0$, and $f_4=-\frac{K_{10}}{24}$.

Use $T_1$ determined by \eqref{E:T12} to partially normalise $H$, with condition \eqref{E:T1 invertible} and 
\begin{align}
	\label{E:Second conjugation nonself v=2}
	\left\{\begin{aligned}
		s_{-1,1}&=0;s_{-1,2}=-s_{-1,0};
		\\s_{-2,0}&=\frac{v_0}{4};s_{-2,1}=(-\frac{1}{8}+\frac{I}{8}) v_0;
		\\s_{-3,0}&=\frac{1}{4} v_0 s_{-1,0};
		\\ s_{-1,0}&=1;
		\\s_{0,2}&= \frac{3 f_3+2 I}{3 f_3-2 I}\quad \text{or}\quad 0;
		\end{aligned}\right\}
\end{align}

We have two possibilities:
\begin{enumerate}
	\item $f_4\neq 0$, hence $K_{10}\neq 0$, we have 
\begin{multline*}
	H=D^6+\frac{1+I}{4}K_{12}(-IA_1+A_3)D^2+(A_1-A_3)D
	\\+\frac{-1+I}{8}K_{10}(A_1+IA_3)+\frac{3K_{12}^2}{2K_{10}}(A_1-A_3)D^{-1}
	\\+(\frac{1}{16}-\frac{I}{16}) K_{12}^2 ((1+I)-A_1-I A_3)D^{-2}
\end{multline*}

Since $K_{14}= 0,K_{10}\neq 0$ in this case, we know the sheaf $\mathcal{F}\cong \mathcal{A}\otimes \mathcal{O}(q)$, with $q=q_1=q_2=(-\frac{K_{12}}{2},\frac{K_{10}}{2})$. The matrix form looks like
$$
\left(
\begin{array}{cccc}
	-\frac{K_{10}}{4} & 0 & D^4+\frac{K_{12}}{2} & 0 \\
	\frac{3 I K_{12}^2}{K_{10}} & -\frac{K_{10}}{4} & 2 I & D^4+\frac{K_{12}}{2} \\
	\frac{1}{4} \left(4 D^8-2 D^4 K_{12}+K_{12}^2\right) & 0 & \frac{K_{10}}{4} & 0 \\
	-2 I \left(D^4-K_{12}\right) & \frac{1}{4} \left(4 D^8-2 D^4 K_{12}+K_{12}^2\right) & -\frac{3 I K_{12}^2}{K_{10}} & \frac{K_{10}}{4} \\
\end{array}
\right)
$$
\item $f_4=0$, then $K_{10}=0$, then the curve $X$ is singular of cuspidal case. So that $K_{12}=0$, and $v_0=\frac{-10f_5}{f_3}$. We have
$$
H=D^6+(A_1-A_3)D
$$
$\mathcal{F}\cong \mathcal{U}$ with $q=s=(0,0)$, the matrix form looks like 
$$
\left(
\begin{array}{cccc}
	0 & 0 & D^4 & 0 \\
	0 & 0 & 2 I & D^4 \\
	D^8 & 0 & 0 & 0 \\
	-2 I D^4 & D^8 & 0 & 0 \\
\end{array}
\right)
$$

\end{enumerate}

\subsubsection{$\nu= 3$ case}

Assume $\nu=3$, then $f_1=f_2=f_3=0,f_4\neq 0$. Look at the coefficient of $H^2-D^{12}+a=0$ in $A_2$ of $D^{-1}$, we get $f_5=0$. Consider \eqref{E:Gr non self}, here 
\begin{multline*}
	K_{14}-2K_{10}f+6K_{12}f^2+2K_{11}f^3-f^4+f''^2-2f'f'''
	\\=(K_{10} K_{11}+3 K_{12}^2)-2f_4(K_{10}+24f_4)x^4+(\cdots)x^6+\cdots
\end{multline*}
with
$$
2f'^2=32f_4^2x^6+\cdots
$$
We know that $K_{14}=0$ with $f_4=-\frac{K_{10}}{24}$. Use $T_1, T_2$ determined by \eqref{E:T12} to partially normalise $H$, with  condition \eqref{E:T1 invertible} and 
\begin{align}
	\label{E:Second conjugation nonself v=3}
	\left\{\begin{aligned}
		s_{-1,1}&=0;s_{-1,2}=-s_{-1,0};
		\\s_{-2,0}&=\frac{v_0}{4};s_{-2,1}=(-\frac{1}{8}+\frac{I}{8}) v_0;
		\\s_{-3,0}&=\frac{1}{4} v_0 s_{-1,0};
		\\ s_{-1,0}&=1;
		\\s_{0,2}&=  0;
		\end{aligned}\right\}
\end{align}
The modified  $H$ is 
\begin{multline*}
	H=D^6+\frac{1+I}{4}K_{12}(-IA_1+A_3)D^2
	\\+\frac{-1+I}{8}K_{10}(A_1+IA_3)+(\frac{1}{16}-\frac{I}{16}) K_{12}^2 ((1+I)-A_1-I A_3)D^{-2}
\end{multline*}

$f_4\neq 0$ so that $K_{10}\neq 0$, here $\nu=3$, then $\mathcal{F}\cong \mathcal{O}(q)\oplus \mathcal{O}(q)$, with $q=q_1=q_2=(-\frac{K_{12}}{2},\frac{K_{10}}{2})$. The matrix form looks like:
	$$
	\left(
	\begin{array}{cccc}
		-\frac{K_{10}}{4} & 0 & D^4+\frac{K_{12}}{2} & 0 \\
		0 & -\frac{K_{10}}{4} & 0 & D^4+\frac{K_{12}}{2} \\
		D^8-\frac{D^4 K_{12}}{2}+\frac{K_{12}^2}{4} & 0 & \frac{K_{10}}{4} & 0 \\
		0 & D^8-\frac{D^4 K_{12}}{2}+\frac{K_{12}^2}{4} & 0 & \frac{K_{10}}{4} \\
	\end{array}
	\right)
	$$

\subsection{The relation of normal forms corresponding to isomorphic sheaves}
\label{S:last_subsection}

According to theorem \ref{T:parametrisation} we know that  partially normalised normal forms in matrix form which correspond to isomorphic sheaves  should be conjugated by a block diagonal matrix. Let $H, H'$ be such normal forms,  $\hat{H}:=\psi\circ\hat{\Phi}(H),\hat{H'}:=\psi\circ\hat{\Phi}(H')$ are their matrix presentation, then 
$$
\hat{H'}=T^{-1}\hat{H}T.
$$

In this section we find all the transfer matrices $T$ for the list of partially normalised normal forms corresponding to different sheaves calculated in previous sections.  Let's collect them first:
\begin{enumerate}
	\item $\mathcal{F}\cong \mathcal{S}\oplus \mathcal{S}$. This only happens in self-adjoint case when $\nu = 3$, and the curve is singular.
	$$\hat{H}=\left(
	\begin{array}{cccc}
		0 & 0 & D^4 & 0 \\
		0 & 0 & 0 & D^4 \\
		D^8 & 0 & 0 & 0 \\
		0 & D^8 & 0 & 0 \\
	\end{array}
	\right)$$
	\item $\mathcal{F}\cong B_q$. This only happens in self-adjoint case, the curve is singular, with two possibilities: $w_1\neq 0,w_0=a=0$ or $w_1=w_2=a=0,w_3\neq 0$. Denote 
	$$
	\hat{H}_1=\left(
	\begin{array}{cccc}
		0 & -2 & D^4 & 0 \\
		0 & 0 & 0 & D^4 \\
		D^8 & 0 & 0 & 2 \\
		0 & D^8 & 0 & 0 \\
	\end{array}
	\right)
	$$
	$$
	\hat{H}_2=\left(
	\begin{array}{cccc}
		0 & 0 & D^4 & 0 \\
		2 & 0 & 0 & D^4 \\
		D^8 & 0 & 0 & 0 \\
		0 & D^8 & -2 & 0 \\
	\end{array}
	\right)
	$$
	The transfer matrix is 
	$$
	T=\left(
	\begin{array}{cccc}
		0 & 1 & 0 & 0 \\
		-1 & 0 & 0 & 0 \\
		0 & 0 & 0 & 1 \\
		0 & 0 & -1 & 0 \\
	\end{array}
	\right)
	$$
	with $\hat{H}_1=T^{-1}\hat{H}_2T$.
	\item $\mathcal{F}\cong \mathcal{O}(q_1)\oplus \mathcal{O}(q_2)$. The curve is not singular, it can happen both in self-adjoint case, and non self-adjoint case:
	
	This one occurs in self-adjoint $\nu=0,8a+w_0^3\neq 0$ case, where $q_+=(-\frac{w_0}{2},\frac{\sqrt{-8a-w_0^3}}{2}),q_-=(-\frac{w_0}{2},\frac{-\sqrt{-8a-w_0^3}}{2})$.
	$$
	\hat{H}_1=\left(
	\begin{array}{cccc}
		0 & -2 & D^4+\frac{w_0}{2} & 0 \\
		\frac{1}{16} (8 a+w_0^3) & 0 & 0 & D^4+\frac{w_0}{2} \\
		\frac{1}{4} ( 4D^8-2 w_0D^4+w_0^2)& 0 & 0 & 2 \\
		0 & \frac{1}{4} \left(4 D^8-2 w_0D^4 +w_0^2\right) & \frac{1}{16} \left(-8 a-w_0^3\right) & 0 \\
	\end{array}
	\right)
	$$
	This one occurs in self-adjoint $\nu=1,8a+w_0^3\neq 0$ case, where $q_+=(-\frac{w_0}{2},\frac{\sqrt{-8a-w_0^3}}{2}),q_-=(-\frac{w_0}{2},\frac{-\sqrt{-8a-w_0^3}}{2})$.
	$$
	\hat{H}_2=\left(
	\begin{array}{cccc}
		\frac{\sqrt{-8 a-w_0^3}}{2 \sqrt{2}} & 0 & D^4+\frac{w_0}{2} & 0 \\
		0 & -\frac{\sqrt{-8 a-w_0^3}}{2 \sqrt{2}} & 0 & D^4+\frac{w_0}{2} \\
		D^8-\frac{1}{2} D^4 w_0+\frac{w_0^2}{4} & 0 & -\frac{\sqrt{-8 a-w_0^3}}{2 \sqrt{2}} & 0 \\
		0 & D^8-\frac{1}{2} D^4 w_0+\frac{w_0^2}{4} & 0 & \frac{\sqrt{-8 a-w_0^3}}{2 \sqrt{2}} \\
	\end{array}
	\right)
	$$
	This one occurs in non self-adjoint case, $\nu=0,K_{14}\neq 0$, where $q_{1,2}=\frac{1}{2} (\sqrt{-K_{14}}-K_{12}),\frac{1}{2} (K_{10}\pm K_{11} \sqrt{-K_{14}}))$. Here we use the same notations of $G_1,G_2,G_3,\Omega_1,\Omega_2,G_7,G_8,G_9$ like the discussion before:
    $$\hat{H}_3=\begin{pmatrix}
		G_1	& G_2 \\
		G_3	& -G_1
	\end{pmatrix}$$

	This one occurs in non self-adjoint case, $\nu=1,K_{14}\neq 0$, where $q_{1,2}=\frac{1}{2} (\sqrt{-K_{14}}-K_{12}),\frac{1}{2} (K_{10}\pm K_{11} \sqrt{-K_{14}}))$ and $f_2=\pm\frac{1}{2}\sqrt{-K_{14}}$.
	$$
	\hat{H}_4=\begin{pmatrix}
		G_7	& G_8 \\
		G_9	& -G_7
	\end{pmatrix}
	$$
	
	Notice that the points $q_{\pm}$ in self-adjoint case can't match up with $q_{1,2}$ in non self-adjoint case: 
	if $\{q_{\pm}\}=q_{1,2}$, we may assume $q_+=q_1,q_-=q_2$, then $K_{10}=0$, hence by $\frac{1}{2} (\sqrt{-K_{14}}-K_{12})=-\frac{w_0}{2}$, we get
	$$
	K_{12}=\frac{1-\sqrt{3}I}{4}w_0
	$$
	
	The transfer matrix is 
	$$
	T_1=\left(
	\begin{array}{cccc}
		-\frac{\sqrt{-8 a-w_0^3}}{4 \sqrt{2}} & 1 & 0 & 0 \\
		\frac{1}{2} & \frac{2 \sqrt{2}}{\sqrt{-8 a-w_0^3}} & 0 & 0 \\
		0 & 0 & -\frac{\sqrt{-8 a-w_0^3}}{4 \sqrt{2}} & 1 \\
		0 & 0 & \frac{1}{2} & \frac{2 \sqrt{2}}{\sqrt{-8 a-w_0^3}} \\
	\end{array}
	\right)
	$$
	with $T_1^{-1}\hat{H}_2T=\hat{H}_1$.
	
	and 
	$$
	T_2=\left(
	\begin{array}{cccc}
		1 & \frac{2}{\sqrt{-K_{14}}} & 0 & 0 \\
		\frac{1}{4} \left(\sqrt{-K_{14}}+I K_{12}\right) & -\frac{\sqrt{-K_{14}}-I K_{12}}{2 \sqrt{-K_{14}}} & 0 & 0 \\
		0 & 0 & 1 & \frac{2}{\sqrt{-K_{14}}} \\
		0 & 0 & \frac{1}{4} \left(\sqrt{-K_{14}}+I K_{12}\right) & -\frac{\sqrt{-K_{14}}-I K_{12}}{2 \sqrt{-K_{14}}} \\
	\end{array}
	\right)
	$$
	with $T_2^{-1}\hat{H_3}T_2=\hat{H}_4$.
	\item $\mathcal{F}\cong \mathcal{O}(q)\oplus \mathcal{O}(q)$, it can happen both in self-adjoint case and non self-adjoint case with $\nu=3$.
	
	This one occurs in self-adjoint case, $q_+=q_-=q=(-\frac{w_0}{2},0)$, with $8a+w_0^3=0$
	$$
	\hat{H}_1=\left(
	\begin{array}{cccc}
		0 & 0 & D^4-\sqrt[3]{a} & 0 \\
		0 & 0 & 0 & D^4-\sqrt[3]{a} \\
		a^{2/3}+\sqrt[3]{a} D^4+D^8 & 0 & 0 & 0 \\
		0 & a^{2/3}+\sqrt[3]{a} D^4+D^8 & 0 & 0 \\
	\end{array}
	\right)
	$$
	
	This one comes from non self-adjoint case $K_{14}=0, K_{10}\neq 0$, $q=q_1=q_2=(-\frac{K_{12}}{2},\frac{K_{10}}{2})$,
	$$
	\hat{H}_2=\left(
	\begin{array}{cccc}
		-\frac{K_{10}}{4} & 0 & D^4+\frac{K_{12}}{2} & 0 \\
		0 & -\frac{K_{10}}{4} & 0 & D^4+\frac{K_{12}}{2} \\
		D^8-\frac{D^4 K_{12}}{2}+\frac{K_{12}^2}{4} & 0 & \frac{K_{10}}{4} & 0 \\
		0 & D^8-\frac{D^4 K_{12}}{2}+\frac{K_{12}^2}{4} & 0 & \frac{K_{10}}{4} \\
	\end{array}
	\right)
	$$
	We see here $\hat{H}_1,\hat{H}_2$ match only when $K_{10}=0$ (so it can't happen).
	\item $\mathcal{F}\cong\mathcal{A}\otimes \mathcal{O}(q)$, it can happen both in self-adjoint case and non self-adjoint case with $\nu\neq 3$.
	
	This comes from self-adjoint case $\nu=0,w_0^3+8a=0$, $q=(-\frac{w_0}{2},0)$,
	$$
	\hat{H}_1=\left(
	\begin{array}{cccc}
		0 & -2 & D^4-\sqrt[3]{a} & 0 \\
		0 & 0 & 0 & D^4-\sqrt[3]{a} \\
		a^{2/3}+\sqrt[3]{a} D^4+D^8 & 0 & 0 & 2 \\
		0 & a^{2/3}+\sqrt[3]{a} D^4+D^8 & 0 & 0 \\
	\end{array}
	\right)
	$$
	This comes from self-adjoint case, $\nu=2,w_0^3+8a=0$ $q=(-\frac{w_0}{2},0)$,
	$$
	\hat{H}_2=\left(
	\begin{array}{cccc}
		0 & 0 & D^4-\sqrt[3]{a} & 0 \\
		2 & 0 & 0 & D^4-\sqrt[3]{a} \\
		a^{2/3}+\sqrt[3]{a} D^4+D^8 & 0 & 0 & 0 \\
		0 & a^{2/3}+\sqrt[3]{a} D^4+D^8 & -2 & 0 \\
	\end{array}
	\right)
	$$
	
	This comes from non self-adjoint case, $\nu=0,K_{14}=0,K_{10}\neq 0$, $q_1=q_2=(-\frac{K_{12}}{2},\frac{K_{10}}{2})$, here use the same notations $G_4,G_5,G_6$ as before
	$$\hat{H}_3=\begin{pmatrix}
		G_4	& G_5 \\
		G_6	& -G_4
	\end{pmatrix}$$
	
	This comes from non self-adjoint case, $\nu=2,K_{14}=0,K_{10}\neq 0$, $q=q_1=q_2=(-\frac{K_{12}}{2},\frac{K_{10}}{2})$,
	$$
	\hat{H}_4=\left(
	\begin{array}{cccc}
		-\frac{K_{10}}{4} & 0 & D^4+\frac{K_{12}}{2} & 0 \\
		\frac{3 I K_{12}^2}{K_{10}} & -\frac{K_{10}}{4} & 2 I & D^4+\frac{K_{12}}{2} \\
		\frac{1}{4} \left(4 D^8-2 D^4 K_{12}+K_{12}^2\right) & 0 & \frac{K_{10}}{4} & 0 \\
		-2 I \left(D^4-K_{12}\right) & \frac{1}{4} \left(4 D^8-2 D^4 K_{12}+K_{12}^2\right) & -\frac{3 I K_{12}^2}{K_{10}} & \frac{K_{10}}{4} \\
	\end{array}
	\right)
	$$
	We can see here the self-adjoint case can't match up with non self-adjoint case since $K_{10}\neq 0$. The relation within self-adjoint and non self-adjoint are as follows:
	$$
	T_1=\left(
	\begin{array}{cccc}
		0 & 1 & 0 & 0 \\
		-1 & 0 & 0 & 0 \\
		0 & 0 & 0 & 1 \\
		0 & 0 & -1 & 0 \\
	\end{array}
	\right)
	$$
	with $T_1^{-1}\hat{H}_2T_1=\hat{H}_1$.
	
	In non self-adjoint case, since $K_{14}=3K_{12}^2+K_{10}K_{11}=0$ but $K_{10}\neq 0$, so that $K_{12}\neq 0$, we have 
	$$
	T_2=\left(
	\begin{array}{cccc}
		\frac{\left(\frac{1}{4}+\frac{I}{4}\right) K_{12}}{\sqrt{2}} & -\frac{1-I}{\sqrt{2}} & 0 & 0 \\
		0 & -\frac{(2-2 I) \sqrt{2}}{K_{12}} & 0 & 0 \\
		0 & 0 & \frac{\left(\frac{1}{4}+\frac{I}{4}\right) K_{12}}{\sqrt{2}} & -\frac{1-I}{\sqrt{2}} \\
		0 & 0 & 0 & -\frac{(2-2 I) \sqrt{2}}{K_{12}} \\
	\end{array}
	\right)
	$$
	with $T_2^{-1}\hat{H}_4T_2=\hat{H}_3$.
	\item $\mathcal{F}\cong\mathcal{U}$ case, the curve is singular, it can only happen in non self-adjoint case when $K_{14}=K_{10}=0$.
	
	This one comes from $\nu=0$ case
	$$
	\hat{H}_1=\left(
	\begin{array}{cccc}
		0 & 0 & D^4 & 2 \\
		0 & 0 & 0 & D^4 \\
		D^8 & -2 D^4 & 0 & 0 \\
		0 & D^8 & 0 & 0 \\
	\end{array}
	\right)
	$$
	
	This one comes from $\nu=2$ case
	$$
	\hat{H}_2=\left(
	\begin{array}{cccc}
		0 & 0 & D^4 & 0 \\
		0 & 0 & 2 I & D^4 \\
		D^8 & 0 & 0 & 0 \\
		-2 I D^4 & D^8 & 0 & 0 \\
	\end{array}
	\right)
	$$
	
	The relation matrix is 
	$$
	T=\left(
	\begin{array}{cccc}
		0 & 1 & 0 & 0 \\
		I & 0 & 0 & 0 \\
		0 & 0 & 0 & 1 \\
		0 & 0 & I & 0 \\
	\end{array}
	\right)
	$$
	with $T^{-1}\hat{H}_2T=\hat{H}_1$.
\end{enumerate}

\noindent J. Guo,  School of Mathematical Sciences, Peking University and Sino-Russian Mathematics Center,  Beijing, China 
\\ 
\noindent\ e-mail:
$123281697@qq.com$

\vspace{0.5cm}

\noindent A. Zheglov, Lomonosov Moscow State  University, Faculty
of Mechanics and Mathematics, Department of differential geometry
and applications, Leninskie gory, GSP, Moscow, \nopagebreak 119899,
Russia\\
National Research University Higher School of Economics\\
\\ \noindent e-mail
 $alexander.zheglov@math.msu.ru$, $abzv24@mail.ru$

\end{document}